\documentclass[a4paper]{article}

\usepackage{ifthen}
\usepackage{amsmath}
\usepackage{amssymb}
\usepackage{amsthm}
\usepackage{bm}
\usepackage{enumerate}
\usepackage{textcomp}
\usepackage{pifont}
\usepackage{stmaryrd}
\usepackage{wasysym}
\usepackage{epic}
\usepackage{eepic}
\input xy
\xyoption{all}
\usepackage{rotating}

\newfont{\ffi}{cmfi10 scaled 1000}

\newcommand{\mc}[1]{{\mathcal{#1}}}
\newcommand{\mf}[1]{{\mathfrak{#1}}}
\newcommand{\bb}[1]{{\mathbb{#1}}}

\DeclareMathOperator{\dom}{dom}
\DeclareMathOperator{\spn}{span}
\DeclareMathOperator{\id}{id}
\DeclareMathOperator{\clos}{clos}
\DeclareMathOperator{\Sub}{Sub}
\DeclareMathOperator{\rba}{rba}
\DeclareMathOperator{\Adm}{Adm}
\DeclareMathOperator{\mt}{mt}
\newcommand{\PW}{{\mc P\hspace*{-1pt}W\!}}
\DeclareMathOperator{\Min}{Min}

\newcommand{\qu}{\overline}
\newcommand{\comment}[1]{}

\DeclareMathOperator{\RE}{Re}
\DeclareMathOperator{\IM}{Im}
\renewcommand{\Re}{\RE}
\renewcommand{\Im}{\IM}

\newlength{\maxlabwidth}
\newenvironment{axioms}[1]{
    \setlength{\maxlabwidth}{#1}
    \begin{list}{}{
    \setlength{\rightmargin}{2mm}
    \setlength{\leftmargin}{\maxlabwidth}\addtolength{\leftmargin}{2mm}
    \setlength{\labelsep}{0mm}
    \setlength{\labelwidth}{\maxlabwidth}
    \setlength{\itemindent}{0mm}
    
    }
    }{
    \end{list}
    }

\numberwithin{equation}{section}

\swapnumbers
\theoremstyle{plain}
	\newtheorem{lem}{Lemma}[section]
	\newtheorem{pro}[lem]{Proposition}
	\newtheorem{thm}[lem]{Theorem}
	\newtheorem{cor}[lem]{Corollary}
	\newtheorem{namth}[lem]{}
\theoremstyle{definition}
	\newtheorem{defi}[lem]{Definition}
\theoremstyle{remark}
	\newtheorem{rem}[lem]{Remark}
	\newtheorem{exa}[lem]{Example}
	\newtheorem{namre}[lem]{}
\renewcommand{\qedsymbol}{\raisebox{-2pt}{\large\ding{113}}}
\newcommand{\defendsymbol}{}
\newcommand{\qedsymbolsave}{\qedsymbol}
\newenvironment{lemma}[2][]{\begin{lem}[#1]\label{LE#2}}{\end{lem}}
\newcommand{\leref}[1]{Lemma \ref{LE#1}}
\newenvironment{proposition}[2][]{\begin{pro}[#1]\label{PR#2}}{\end{pro}}
\newcommand{\prref}[1]{Proposition \ref{PR#1}}
\newenvironment{theorem}[2][]{\begin{thm}[#1]\label{TH#2}}{\end{thm}}
\newcommand{\thref}[1]{Theorem \ref{TH#1}}
\newenvironment{corollary}[2][]{\begin{cor}[#1]\label{CO#2}}{\end{cor}}
\newcommand{\coref}[1]{Corollary \ref{CO#1}}

\newenvironment{definition}[2][]{\begin{defi}[#1]\label{DE#2}}{
	\renewcommand{\qedsymbolsave}{\qedsymbol}\renewcommand{\qedsymbol}{\defendsymbol}
	\popQED{\qed}\renewcommand{\qedsymbol}{\qedsymbolsave}\end{defi}}

\newenvironment{remark}[2][]{\begin{rem}[#1]\label{RE#2}}{
	\renewcommand{\qedsymbolsave}{\qedsymbol}\renewcommand{\qedsymbol}{\defendsymbol}
	\popQED{\qed}\renewcommand{\qedsymbol}{\qedsymbolsave}\end{rem}}
\newcommand{\reref}[1]{Remark \ref{RE#1}}
\newenvironment{example}[2][]{\begin{exa}[#1]\label{EX#2}}{
	\renewcommand{\qedsymbolsave}{\qedsymbol}\renewcommand{\qedsymbol}{\defendsymbol}
	\popQED{\qed}\renewcommand{\qedsymbol}{\qedsymbolsave}\end{exa}}

\newenvironment{proofof}[1]{\begin{proof}[\textit{Proof (of #1)}]}{\end{proof}}
\newcommand{\bibi}[5]{\bibitem[#5]{#1} \textsc{#2}:\ \textit{#3,}\ {#4.}}
\begin{document}

{\Large\bf
\begin{flushleft}
	Majorization in de~Branges spaces II. Banach spaces generated by majorants
\end{flushleft}
}
\vspace*{3mm}
\begin{center}
	{\sc Anton Baranov, Harald Woracek}
\end{center}

\begin{abstract}
	\noindent
	This is the second part in a series dealing with subspaces of de~Branges spaces 
	of entire function generated 
	by majorization on subsets of the closed upper half-plane. In this part we 
	investigate certain Banach spaces generated by admissible majorants. We study their interplay with 
	the original de~Branges space structure, and their geometry. In particular, we will show that, 
	generically, they will be nonreflexive and nonseparable. 
\end{abstract}
\begin{flushleft}
   {\small
   {\bf AMS Classification Numbers:} 46E15, 46B26, 46E22
   }
\end{flushleft}



%
%
\section{Introduction}

A de~Branges space $\mc H$ is a Hilbert space whose elements are entire functions, and which 
has the following properties: 
\begin{axioms}{14mm}
\item[dB1] For each $w\in\bb C$ the point evaluation $F\mapsto F(w)$ is a continuous linear functional on $\mc H$. 
\item[dB2] If $F\in\mc H$, also $F^\#(z):=\overline{F(\bar z)}$ belongs to $\mc H$ and $\|F^\#\|=\|F\|$.
\item[dB3] If $w\in\bb C\setminus\bb R$ and $F\in\mc H$, $F(w)=0$, then 
	\[
		\frac{z-\bar w}{z-w}F(z)\in\mc H\quad \text{ and }\quad \Big\|\frac{z-\bar w}{z-w}F(z)\Big\|
		=\big\|F\big\|
		\,.
	\]
\end{axioms}
Alternatively, de~Branges spaces can be defined via Hermite--Biehler functions. These are entire functions $E$ 
which satisfy: 
\begin{axioms}{14mm}
\item[HB] For all $z$ in the open upper half-plane $\bb C^+$, we have 
$|E(\qu z)| < |E(z)|$. 
\end{axioms}
For a Hermite--Biehler function $E$ define 
\[
	\mc H(E):=\Big\{F\text{ entire}:\,\frac FE,\frac{F^\#}E\in H^2(\bb C^+)\Big\}
	\,,
\]
\[
	\|F\|_{\mc H(E)}:=\Big(\int_{\bb R}\Big|\frac{F(t)}{E(t)}\Big|^2\,dt\Big)^{\frac 12},\quad F\in\mc H(E)
	\,,
\]
where $H^2(\bb C)$ denotes the Hardy space in the upper half-plane. 
Then $\mc H(E)$ is a de~Branges space. Conversely, every de~Branges space can be obtained in this way, cf.\ 
\cite{debranges:1968}. 

In the theory of de~Branges spaces, an outstandingly important role is played by their de~Branges subspaces 
(dB-subspaces, for short). These are those subspaces $\mc L$ of a de~Branges space $\mc H$ 
which are themselves de~Branges spaces with the norm inherited from $\mc H$. 

In \cite{baranov.woracek:spmaj} we have investigated a general procedure to construct dB-subspaces of a given 
de~Branges space $\mc H$ by means of majorization. For a function $\mf m:D\to[0,\infty)$, defined on some subset 
$D$ of the closed upper half-plane $\bb C^+\cup\bb R$, we have defined 
\[
	R_{\mf m}(\mc H):=\big\{\,F\in\mc H:\ \exists\,C>0:|F(z)|,|F^\#(z)|\leq C\mf m(z),
	\ z\in D\,\big\}
	\,,
\]
and 
\[
	\mc R_{\mf m}(\mc H):=\clos_{\mc H} R_{\mf m}(\mc H)
	\,.
\]
Provided $R_{\mf m}(\mc H)\neq\{0\}$ and $\mf m$ satisifies a mild regularity condition, 
the space $\mc R_{\mf m}(\mc H)$ is a dB-subspace of $\mc H$, cf.\ 
\cite[Theorem 3.1]{baranov.woracek:spmaj}. 
In this case we say that $\mf m$ is an admissible majorant for $\mc H$; 
the set of all admissible majorants is denoted by $\Adm\mc H$
(see \cite[Definition 3.3]{baranov.woracek:spmaj}). The main task in \cite{baranov.woracek:spmaj} was to investigate which 
dB-subspaces of $\mc H$ can be represented as $\mc R_{\mf m}(\mc H)$. Results, of course, depend on the set $D$ where 
majorization is permitted. We showed that every dB-subspace $\mc L$ of $\mc H$ is of the form 
$\mc R_{\mf m}(\mc H)$ when $D$ is sufficiently large, and obtained a number of results on representability 
by specific majorants defined on specific (smaller) subsets $D$.

Starting point for the present paper is the following notice: 
Those elements of a space $\mc R_{\mf m}(\mc H)$ about which one has explicit information, are the elements of 
$R_{\mf m}(\mc H)$. Hence, a closer investigation of $R_{\mf m}(\mc H)$, rather than just of $\mc R_{\mf m}(\mc H)$, 
is desirable. 

On the space $R_{\mf m}(\mc H)$ a stronger norm than $\|.\|_{\mc H}$ can defined in a natural way, namely as 
\begin{multline*}
	\|F\|_{\mf m}:=\max\big\{\|F\|_{\mc H},\min\{C\geq 0:\,|F(z)|,|F^\#(z)|\leq C\mf m(z),z\in D\}\big\},\\
	F\in R_{\mf m}(\mc H)\,.
\end{multline*}
It is seen with a routine argument that $R_{\mf m}(\mc H)$, if endowed with the norm $\|.\|_{\mf m}$, becomes a Banach space. 
Although quite simple, this fact has interesting consequences and gives rise to some intriguing geometric problems. 
The reason which makes the structure of $\langle R_{\mf m}(\mc H),\|.\|_{\mf m}\rangle$ involved might be explained as follows: 
On the one hand $R_{\mf m}(\mc H)$ is fairly small as a set; it is a subset of the 
Hilbert space $\mc H$. On the other, up to some extent, the norm $\|.\|_{\mf m}$ behaves badly; it involves an $L^\infty$-component. 

Let us describe the results and organization of the present paper. In Section 2, after providing some basics, 
we present two instances of the interaction between the de~Branges space structure of $\mc H$ and the Banach space 
structure of $R_{\mf m}(\mc H)$. Namely, we show that the maximal rate of exponential growth of functions in 
$\mc R_{\mf m}(\mc H)$ is already attained within $R_{\mf m}(\mc H)$, and that reflexivity of 
$\langle R_{\mf m}(\mc H),\|.\|_{\mf m}\rangle$ 
implies its separability, cf.\ \prref{R2} and \prref{R3}. The proofs of these results 
are not difficult, but nicely illustrate 
the interplay of $\|.\|_{\mc H}$ and $\|.\|_{\mf m}$. 

Section 3 is the most involved part of the present paper. There we discuss the geometry of the 
Banach space $\langle R_{\mf m}(\mc H),\|.\|_{\mf m}\rangle$ for a particular majorant and two 
particular domains of majorization. Namely, if $\mc L=\mc H(E_1)$ is a dB-subspace of $\mc H$, 
we consider the majorant $\mf m_{E_1}|_D$ where 
\[
	\mf m_{E_1}(z):=\frac{|E_1(z)|}{|z+i|}\quad\text{ and }\quad D:=i[1,\infty)\text{ or }D:=\bb R
	\,.
\]
This majorant already has been used and investigated intensively in \cite{baranov.woracek:spmaj}. Although it is 
probably one of the simplest majorants one can think of, it is already quite hard to obtain 
knowledge on $R_{\mf m}(\mc H)$. It turns out that the geometric structure of $R_{\mf m_{E_1}|_D}(\mc H)$ varies 
from very simple to extremely complicated, and is closely related to the 
distribution of zeros of $E_1$ or, more generally, the behaviour of the inner function $E_1^{-1}E_1^\#$. 
In case $D=i[1\infty)$, roughly speaking, the zeros of $E$ which are close to $\bb R$ 
give "simple" parts of the space, whereas zeros separated from the real axis give "complicated" parts of the space, cf.\ 
\thref{R17} and \thref{R7}. A good illustration of this idea is also \coref{R16}. The case $D=\bb R$ is different; 
it turns out that the geometric structure of $R_{\mf m}(\mc H)$ will always be complicated, cf.\ \thref{R21}. 

In the last section of the paper we revisit the question of representability of dB-subspaces by means of majorization, taking 
up the refined viewpoint of the space $R_{\mf m}(\mc H)$. It is a consequence of the Banach space structure of 
$R_{\mf m}(\mc H)$ that, for each given majorant $\mf m$, 
there exists a (in some sense) smallest one $\mf m^\flat$ among all the majorants $\mf m_1$ with 
$R_{\mf m_1}(\mc H)=R_{\mf m}(\mc H)$. This majorant is fairly smooth and reflects many properties of $\mc R_{\mf m}(\mc H)$, cf.\ 
\prref{R37}. Moreover, the notion of $\mf m^\flat$ can be used to characterize minimal elements in the set of all admissible 
majorants; a topic studied for majorization along $\bb R$ e.g.\ in \cite{baranov.havin:2006} or 
\cite{havin.mashregi:2003a}. It turns out that minimal majorants correspond 
to one-dimensional dB-subspaces representable as $\mc R_{\mf m}(\mc H)$. In conjunction with our previous results on representability, 
this fact yields criteria for existence of minimal majorants, cf.\ \coref{R41}. 

The notation in the present paper will follow the notation 
introduced in \cite{baranov.woracek:spmaj}. Also, we will 
use without further notice the basics of the theory of 
de~Branges spaces as compiled in the preliminaries section of 
that paper. Moreover, since in the present context this is no loss in generality, we will assume that all 
de~Branges spaces $\mc H$ have the following property: Whenever $x\in\bb R$, there exists an element $F\in\mc H$ with 
$F(x)\neq 0$. Also, bounded sets $D$ give only trivial results, hence we will throughout this paper exclude 
bounded sets from our discussion.

\section{The Banach space $R_{\mf m}(\mc H)$}

The following simple observation is the basis of all considerations 
made in this paper. For this reason we provide an explicit proof. 
Let the set $\Adm\mc H$ be defined as 
in \cite[Definition 3.3]{baranov.woracek:spmaj}. 

\begin{proposition}{R1}
	Let $\mc H$ be a de~Branges space and $\mf m\in\Adm\mc H$. 
	Then $\langle R_{\mf m}(\mc H),\|.\|_{\mf m}\rangle$ is a Banach space. 
\end{proposition}
\begin{proof}
	Let $(F_n)_{n\in\bb N}$, $F_n\in R_{\mf m}(\mc H)$, be a Cauchy sequence with respect to the norm $\|.\|_{\mf m}$. Since 
	$\|.\|_{\mf m}\geq\|.\|_{\mc H}$, it is thus also a Cauchy sequence in $\mc H$. By completeness of 
	$\mc H$, $(F_n)_{n\in\bb N}$ converges with respect to the norm $\|.\|_{\mc H}$, say $F:=\lim_{n\to\infty}F_n\in\mc H$. 
	Set $C:=\sup_{n\in\bb N}\|F_n\|_{\mf m}<\infty$. Since convergence in $\mc H$ implies pointwise convergence, we have 
	\[
		|F(z)|=\lim_{n\to\infty}|F_n(z)|\leq C\mf m(z),\quad z\in D
		\,.
	\]
	Similarly, it follows that $|F^\#(z)|\leq C\mf m(z)$, $z\in D$. Hence $F\in R_{\mf m}(\mc H)$. 
	
	Let $\epsilon>0$ be given, and choose $N\in\bb N$ with $\|F_n-F_m\|_{\mf m}\leq\epsilon$, $n,m\geq N$. Then we have 
	\[
		|F_n(z)-F_m(z)|\leq\epsilon\mf m(z),\quad z\in D,\ n,m\geq N
		\,.
	\]
	Passing to the limit $m\to\infty$, it follows that $|F_n(z)-F(z)|\leq\epsilon\mf m(z)$, $z\in D$, $n\geq N$. 
	Together with convergence in $\mc H$, this implies that $\lim_{n\to\infty}F_n=F$ with respect to the norm $\|.\|_{\mf m}$. 
\end{proof}

As a first step towards getting acquainted with this Banach space, let us discuss its dual. 

\begin{remark}{R4}
	Let $\mf m:D\to[0,\infty)$ belong to $\Adm\mc H$. Assume that 
	$D$ is a closed subset of $\bb C^+\cup\bb R$, and that 
	\begin{equation}\label{R5}
		\forall\,w\in D:\ (z-w)^{-\mf d_{\mf m}(w)}\mf m(z)\ 
		\text{ is continuous at }\, w
		\,.
	\end{equation}
	For the definition of the zero-divisor $\mf d_{\mf m}$ see 
        \cite[Definition 2.1]{baranov.woracek:spmaj}. As it will become clear 
        later, assuming \eqref{R5} is no essential restriction. 
	Denote by $C(D)$ the Banach space of all continuous 
	and bounded functions on $D$ endowed with the 
	supremum norm $\|.\|_{\infty}$. Then the map 
	\[
		\mc J: \left\{
		\begin{array}{rcl}
			\langle R_{\mf m}(\mc H),\|.\|_{\mf m}\rangle & \to & \langle\mc H\times C(D),
				\max\{\|.\|_{\mc H},\|.\|_\infty\}\rangle\\
			F & \mapsto & \big(F,\frac F{\mf m}\big)
		\end{array}
		\right.
	\]
	is an isometric isomorphism of $R_{\mf m}(\mc H)$ onto a closed subspace of $\mc H\times C(D)$. 
	Note that the condition \eqref{R5} ensures that $\mf m^{-1}F$ 
        is a continuous function whenever $\mf d_F\geq\mf d_{\mf m}$. Since $D$ 
	is locally compact, the dual space $C(D)'$ is isomorphic to the space $\rba(D)$ of all regular bounded 
	finitely additive set functions defined on the $\sigma$-algebra of Borel sets on $D$. It follows that 
	\[
		R_{\mf m}(\mc H)'\cong \big(\mc H\times \rba(D)\big)\big/\raisebox{-2pt}{$N$}
		\,,
	\]
	where $N:=\mc J(R_{\mf m}(\mc H))^\perp$. In fact, every continuous linear functional on $R_{\mf m}(\mc H)$ can 
	be written in the form 
	\[
		F\mapsto (F,G)_{\mc H}+\int_DF\,d\mu,\qquad F\in R_{\mf m}(\mc H)
		\,,
	\]
	with some $G\in\mc H$ and $\mu\in\rba(D)$. We also see that the annihilator $N$ is given as 
	\[
		N=\Big\{(G,\mu)\in\mc H\times\rba(D):\,\int_D\frac F{\mf m}\,d\mu=(F,-G)_{\mc H},
		F\in\mc R_{\mf m}(\mc H)\Big\}
		\,.
	\]
\end{remark}

%
%
\begin{flushleft}
	\textbf{Interaction between $\|.\|_{\mf m}$ and $\|.\|_{\mc H}$.}
\end{flushleft}
\vspace*{-2mm}
The interplay between the Banach space structure of $R_{\mf m}(\mc H)$ and the de~Branges space structure of 
$\mc H$ leads to interesting insight. We give two results of this kind. 
First we show that the maximal rate of exponential growth transfers from $R_{\mf m}(\mc H)$ to $\mc R_{\mf m}(\mc H)$. 
For $\alpha\leq 0$, set $\mc H_{(\alpha)}:=\{F\in\mc H:\,\mt_{\mc H}F,\,
\mt_{\mc H}F^\#\leq\alpha\}$. Then $\mc H_{(\alpha)}$ is a closed 
subspace of $\mc H$, cf.\ \cite[Corollary 5.2]{kaltenbaeck.woracek:growth}. Hence we have 
\begin{equation}\label{R43}
	\mt_{\mc H}\mc R_{\mf m}(\mc H)=\sup_{F\in R_{\mf m}(\mc H)}\mt_{\mc H}F
	\,,
\end{equation}
for the definition of $\mt_{\mc H}$ see \cite[Definition 2.1]{baranov.woracek:spmaj}. 
\begin{proposition}{R2}
	Let $\mc H$ be a de~Branges space, and let $\mf m\in\Adm\mc H$. 
	Then there exists a function $F\in R_{\mf m}(\mc H)$, such that 
	\[
		\mt_{\mc H}\mc R_{\mf m}(\mc H)=\mt_{\mc H}F
		\,.
	\]
\end{proposition}
\begin{proof}
	We have $\|.\|_{\mf m}\geq\|.\|_{\mc H}$ on $R_{\mf m}(\mc H)$. 
        Hence, for each $\alpha\leq 0$, the subspace 
	$R_{\mf m}(\mc H)\cap\mc H_{(\alpha)}$ is $\|.\|_{\mf m}$-closed. 
	Consider the value $\alpha:=\sup_{F\in R_{\mf m}(\mc H)}\mt_{\mc H}F$. Then 
	\[
		R_{\mf m}(\mc H)\cap\mc H_{(\beta)}\subsetneq R_{\mf m}(\mc H),\qquad \beta<\alpha
		\,,
	\]
	and hence each of the spaces $R_{\mf m}(\mc H)\cap\mc H_{(\beta)}$, $\beta<\alpha$, is nowhere dense in
	$\langle R_{\mf m}(\mc H),\|.\|_{\mf m}\rangle$. By the Baire Category Theorem the set 
	\[
		R_{\mf m}(\mc H)\setminus\,\bigcup_{n\in\bb N} \Big(R_{\mf m}(\mc H)\cap\mc H_{(\alpha-\frac 1n)}\Big)
	\]
	is dense in $R_{\mf m}(\mc H)$. In particular, there exists $F\in R_{\mf m}(\mc H)$ with $\mt_{\mc H}F=\alpha$. 
\end{proof}

Secondly, we discuss the geometry of $\langle R_{\mf m}(\mc H),\|.\|_{\mf m}\rangle$. Several different topologies 
play a role: 
\begin{enumerate}[$(i)$]
\item The topology $\tau_{lu}$ of locally uniform convergence. 
\item The weak topology $\tau_w$ of $\langle\mc H,\|.\|_{\mc H}\rangle$. 
\item The weak topology $\tau^{\mf m}_w$ of $\langle R_{\mf m}(\mc H),\|.\|_{\mf m}\rangle$. 
\end{enumerate}
Denote by $B_{\mf m}(\mc H)$ the unit ball of the Banach space $\langle R_{\mf m}(\mc H),\|.\|_{\mf m}\rangle$. 
Explicitly, this is 
\[
	B_{\mf m}(\mc H):=\Big\{\ F\in\mc H:\ \ 
	\parbox{47mm}{$\|F\|_{\mc H}\leq 1$ and \\ $|F(z)|,|F^\#(z)|\leq\mf m(z),z\in D$}\Big\}
	\,.
\]

\begin{proposition}{R3}
	Let $\mf m\in\Adm\mc H$. Then the following hold:
	\begin{enumerate}[$(i)$]
	\item Let $B\subseteq\mc H$ be bounded with respect to the norm $\|.\|_{\mc H}$ of $\mc H$. Then $B$ is a normal 
		family of entire functions. We have $\tau_w|_B=\tau_{lu}|_B$. 
	\item The space $\langle R_{\mf m}(\mc H),\|.\|_{\mf m}\rangle$ is reflexive if and only if 
		$\tau^{\mf m}_w|_{B_{\mf m}(\mc H)}=\tau_{lu}|_{B_{\mf m}(\mc H)}$. 
	\item If $\langle R_{\mf m}(\mc H),\|.\|_{\mf m}\rangle$ is reflexive, then it is also separable. 
	\end{enumerate}
\end{proposition}
\begin{proof}
	Let $B$ be a $\|.\|_{\mc H}$-bounded subset of $\mc H$. 
	The weak closure of $B$ is again $\|.\|_{\mc H}$-bounded. Hence, for the proof of $(i)$, we may assume in addition 
	that $B$ is $\tau_w$-closed, and thus $\tau_w$-compact. 
        Let $\nabla_{\!\mc H}$ be defined as 
	\[
		\nabla_{\!\mc H}(z):=\sup\big\{|F(z)|:\,\|F\|_{\mc H}=1\big\}=
		\big(K(z,z)\big)^{1/2},
	\]
	where $K(z,.)$ is the reproducing kernel at the point $z$ 
	(see \cite[Section 2, II, III]{baranov.woracek:spmaj} for details).
	Since $|F(z)|\leq\|F\|_{\mc H}\nabla_{\!\mc H}(z)$ and $\nabla_{\!\mc H}$ is continuous, 
	the family $B$ is locally uniformly bounded, i.e.\ a normal family. 
	Since $\mc H$ is a separable Hilbert space, the restriction $\tau_w|_B$ is metrizable, cf.\ 
	\cite[Theorem 2.6.23]{megginson:1998}. Let $(F_n)_{n\in\bb N}$ be a sequence of elements of $B$ which converges 
	to $F\in B$ with respect to $\tau_w|_B$. Then it converges pointwise and, by the Vitali Theorem, thus also 
	locally uniformly. We see that $\tau_{lu}|_B\subseteq\tau_w|_B$; note here that $\tau_{lu}$ is also 
	metrizable. Since $\tau_w|_B$ is compact and 
	$\tau_{lu}|_B$ is Hausdorff, it follows that actually equality holds, cf.\ \cite[I.9.Corollary 3]{bourbaki:1966}. 
	This finishes the proof of $(i)$. 

	We come to the proof of $(ii)$. The space $\langle R_{\mf m}(\mc H),\|.\|_{\mf m}\rangle$ is reflexive, if and only if 
	its unit ball $B_{\mf m}(\mc H)$ is $\tau^{\mf m}_w$-compact, cf.\ \cite[Theorem 2.8.2]{megginson:1998}. 
	Since $\|.\|_{\mf m}\geq\|.\|_{\mc H}$, we have $\tau_w|_{R_{\mf m}(\mc H)}\subseteq\tau^{\mf m}_w$. 
	The unit ball $B_{\mf m}(\mc H)$ is a $\|.\|_{\mc H}$-bounded subset of $\mc H$. 
	It is $\|.\|_{\mc H}$-closed and convex, and hence 
	also $\tau_w$-closed. Thus $B_{\mf m}(\mc H)$ is $\tau_w$-compact. It follows that 
	$B_{\mf m}(\mc H)$ is $\tau^{\mf m}_w$-compact if and only if $\tau_w|_{B_{\mf m}(\mc H)}=\tau^{\mf m}_w|_{B_{\mf m}(\mc H)}$. 
	By the already proved item $(i)$, the desired assertion follows. 

       	For the proof of $(iii)$, assume that $\langle R_{\mf m}(\mc H),\|.\|_{\mf m}\rangle$ is reflexive. Then, by $(ii)$, the 
	topology $\tau^{\mf m}_w|_{B_{\mf m}(\mc H)}$ equals the topology of locally uniform convergence and, hence, 
	is metrizable. We conclude from \cite[Theorem 2.6.23]{megginson:1998} that $\langle R_{\mf m}(\mc H),\|.\|_{\mf m}\rangle$ 
	is separable. 
\end{proof}

\section{Geometry of $R_{\mf m_{E_1}|_D}(\mc H)$}

Let $\mc H$ be a de~Branges space and let $\mc L\in\Sub^*\mc H$, i.e.\ let $\mc L$ be a dB-subspace of $\mc H$ with 
$\mf d_{\mc L}=\mf d_{\mc H}$ ($=0$), 
cf.\ \cite[Definition 2.4]{baranov.woracek:spmaj} and the paragraph following it. 
Write $\mc L=\mc H(E_1)$, and denote by $K_1(w,z)$ 
the reproducing kernel of $\mc L$. In our previous work two particular majorants were extensively 
investigated. Namely, those obtained by restriction of the functions 
\[
	\nabla_{\!\mc L}(z):=\|K_1(z,.)\|_{\mc H}\ \text{ and }\ \mf m_{E_1}(z):=\frac{|E_1(z)|}{|z+i|}
\]
to the set $D$ under consideration. 
In \cite{baranov.woracek:dbmaj} and \cite{baranov.woracek:spmaj} we showed that: 
\begin{enumerate}[$(i)$]
\item Let $D:=i[1,\infty)$. Then 
	\begin{equation}\label{R42}
		\mc L=\mc R_{\nabla_{\!\mc L}|_D}(\mc H)=\mc R_{\mf m_{E_1}|_D}(\mc H)
		\,.
	\end{equation}
\item Let $D:=\bb R$ and assume that $\mt_{\mc H}\mc L=0$. Then $\mc L=\mc R_{\mf m_{E_1}|_D}(\mc H)$. 
\end{enumerate}
In this section we will study $\langle R_{\mf m}(\mc H),\|.\|_{\mf m}\rangle$ for these situations. 
If $\dim\mc H<\infty$, of course, all questions about the geometry of $R_{\mf m}(\mc H)$ are trivial. Hence we will, 
once and for all, exclude finite dimensional spaces $\mc H$ from our discussion.

%
%
\begin{flushleft}
	\textbf{a. Majorization on the imaginary half-line.}
\end{flushleft}
\vspace*{-2mm}
Consider $D:=i[1,\infty)$. Since $\mc L\subseteq R_{\nabla_{\!\mc L}|_D}(\mc H)$, \eqref{R42} implies that 
the space $R_{\nabla_{\!\mc L}|_D}(\mc H)$ is $\|.\|_{\mc H}$-closed. It follows that the norms 
$\|.\|_{\nabla_{\!\mc L}|_D}$ and $\|.\|_{\mc H}$ both turn $R_{\nabla_{\!\mc L}|_D}(\mc H)$ into a Banach space. However, 
$\|.\|_{\nabla_{\!\mc L}|_D}\geq\|.\|_{\mc H}$, and therefore they are equivalent. 
Being bicontinuously isomorphic to the Hilbert space $\langle\mc L,\|.\|_{\mc H}\rangle$, the geometry of the space 
$\langle R_{\nabla_{\!\mc L}|_D}(\mc H),\|.\|_{\nabla_{\!\mc L}|_D}\rangle$ is very simple. 
For example it is separable, reflexive, has an unconditional basis, etc. 

Things change, when turning to the majorant $\mf m_{E_1}|_D$. Then, as we will show below, the geometry of the space 
$\langle R_{\mf m_{E_1}|_D}(\mc H),\|.\|_{\mf m_{E_1}|_D}\rangle$ varies from very simple to highly complicated. 

For each majorant $\mf m$ we have $R_{\mf m}(\mc H)=R_{\mf m}(\mc R_{\mf m}(\mc H))$. Moreover, trivially, the norms 
$\|.\|_{\mf m}$ of $R_{\mf m}(\mc H)$ and $R_{\mf m}(\mc R_{\mf m}(\mc H))$ are equal. In particular, using \eqref{R42}, 
\[
	R_{\mf m_{E_1}|_D}(\mc H)=R_{\mf m_{E_1}|_D}(\mc L)
\]
including equality of norms. Hence, we may restrict explicit considerations to the case when $\mc L=\mc H$ (and doing so slightly 
simplifies notation). 

Our aim in this subsection is to prove the following two theorems. 
Denote by $\Gamma_\alpha$, $\alpha\in(0,\frac\pi 2)$, the Stolz angle 
\[
	\Gamma_\alpha:=\big\{z\in\bb C:\,\alpha\leq\arg z\leq\pi-\alpha\big\}
	\,.
\]

\begin{theorem}{R17}
	Let $\mc H=\mc H(E)$ be a de~Branges space, and set $\mf m:=\mf m_{E}|_{i[1,\infty)}$. 
	Moreover, denote $\Theta:=E^{-1}E^\#$, and let $(w_n)_{n\in\bb N}$ be the sequence of zeros 
	of $\Theta$ in $\bb C^+$ listed according to their multiplicities. 
	Then the following are equivalent: 
	\begin{itemize}
	\item[$(i)$] The norms $\|.\|_{\mf m}$ and $\|.\|_{\mc H}$ are equivalent on $R_{\mf m}(\mc H)$. 
	\item[$(i')$] The space $R_{\mf m}(\mc H)$ is $\|.\|_{\mc H}$-closed. 
	\item[$(i'')$] We have $R_{\mf m}(\mc H)=\mc H$. 
	\item[$(ii)$] There exists $\varphi\in\bb R$, such that 
		\[
			\sup_{z\in\Gamma_\alpha}\Big(|z|\cdot|e^{i\varphi}-\Theta(z)|\Big)<\infty,\quad 
			\alpha\in\Big(0,\frac\pi 2\Big)
			\,.
		\]
	\item[$(ii')$] There exists $\varphi\in\bb R$ and $\alpha\in\Big(0,\frac\pi 2\Big)$, such that 
		\[
			\liminf_{\substack{|z|\to\infty\\ z\in\Gamma_\alpha}}\Big(|z|\cdot|e^{i\varphi}-\Theta(z)|\Big)<\infty
			\,.
		\]
	\item[$(iii)$] We have $\mt\Theta=0$ and $\sum_{n\in\bb N}\Im w_n<\infty$. 
	\item[$(iii')$] The domain of the multiplication operator $S_{\mc H}$ in $\mc H$ is not dense. 
	\end{itemize}
\end{theorem}

\begin{theorem}{R7}
	Let $\mc H=\mc H(E)$ be a de~Branges space, and set $\mf m:=\mf m_{E}|_{i[1,\infty)}$. 
	Moreover, denote $\Theta:=E^{-1}E^\#$, and let $(w_n)_{n\in\bb N}$ be the sequence of zeros 
	of $\Theta$ in $\bb C^+$ listed according to their multiplicities. 
	Then 
	\[
		(i)\,\Longrightarrow\,(ii)\,\Longrightarrow\,(iii),
	\]
	where $(i)$, $(ii)$, and $(iii)$, are the following conditions: 
	\begin{enumerate}[$(i)$]
	\item We have $\mt\Theta<0$ or $\limsup_{n\to\infty}\Im w_n>0$. 
	\item There exists $\delta>0$, such that 
		\begin{equation}\label{R8}
			\liminf_{\substack{|z|\to\infty\\ \Im z\geq\delta}}|\Theta(z)|=0
			\,.
		\end{equation}
	\item There exists a bicontinuous embedding of $\langle\ell^\infty,\|.\|_\infty\rangle$ into 
		$\langle R_{\mf m}(\mc H),\|.\|_{\mf m}\rangle$. In particular, $R_{\mf m}(\mc H)$ is 
		neither separable nor reflexive. 
	\end{enumerate}
\end{theorem}

\begin{remark}{R6}
	Comparing the condition $(iii)$ of \thref{R17} with $(i)$ of \thref{R7}, 
	it is apparent that these theorems do not establish a full dichotomy; 
	there is a gap between the described situations. At present it is not clear to us what happens in this gap. 
	In particular, it is an open question whether there exists a de~Branges space $\mc H=\mc H(E)$, such that 
	for $\mf m:=\mf m_E|_{i[1,\infty)}$ the space $\langle R_{\mf m}(\mc H),\|.\|_{\mf m}\rangle$ is separable (or reflexive) 
	although $R_{\mf m}(\mc H)\neq\mc H$. 
\end{remark}

We turn to the proof of \thref{R17} and \thref{R7}. 
For the first theorem we will use the following known facts, which can be found,
e.g., in \cite[Theorem 2, Corollary 2]{baranov:2001}. 

\begin{lemma}{R18}
	Let $\Theta$ be an inner function in $\bb C^+$. 
	Then the following are equivalent:
	\begin{enumerate}[$(i)$]
	\item There exists $\varphi\in\bb R$ such that $\sup_{y\geq 1}y|e^{i\varphi}-\Theta(iy)|<\infty$. 
	\item We have $\sup_{y\geq 1}y(1-|\Theta(iy)|)<\infty$. 
	\item There exists $\varphi\in\bb R$ such that $e^{i\varphi}-\Theta\in H^2$. 
	\end{enumerate}
	Assume that $\Theta$ is of the form $\Theta(z)=e^{-iaz}B(z)$ with $a\leq 0$ 
        and $B$ being a Blaschke product. 
	Denote by $(w_n)_{n\in\bb N}$ the sequence of zeros of $B$ listed according to their multiplicities, then 
	the conditions $(i)$--$(iii)$ are further equivalent to 
	\begin{enumerate}[$(i)$]
	\setcounter{enumi}{3}
	\item $a=0$ and $\sum_{n\in\bb N}\Im w_n<\infty$. 
	\end{enumerate}
	\popQED{\qed}
\end{lemma}

\begin{corollary}{R9}
	Let $\mc H=\mc H(E)$ be a de~Branges space. Set $\Theta:=E^{-1}E^\#$, and denote by $(w_n)_{n\in\bb N}$, $w_n\in\bb C^+$, 
	the sequence of zeros of $\Theta$ in $\bb C^+$ listed according to their multiplicities. Then the multiplication 
	operator in $\mc H$ is not densely defined, if and only if $\mt\Theta=0$ and $\sum_{n\in\bb N}\Im w_n<\infty$. 
	\popQED{\qed}
\end{corollary}

\begin{proofof}{\thref{R17}}
	By the just stated corollary we have $(iii)\Leftrightarrow(iii')$. 

	The equivalences $(i)\Leftrightarrow(i')\Leftrightarrow(i'')$ are easy to see: 
	If $\|.\|_{\mf m}$ and $\|.\|_{\mc H}|_{R_{\mf m}(\mc H)}$ are equivalent, then $R_{\mf m}(\mc H)$ is 
	$\|.\|_{\mc H}$-complete and hence also $\|.\|_{\mc H}$-closed. Since in any case $\mc R_{\mf m}(\mc H)=\mc H$, 
	$\|.\|_{\mc H}$-closedness of $R_{\mf m}(\mc H)$ implies that $R_{\mf m}(\mc H)=\mc H$. Finally, if 
	$R_{\mf m}(\mc H)=\mc H$, then $\|.\|_{\mf m}$ and $\|.\|_{\mc H}$ both turn $R_{\mf m}(\mc H)$ into a Banach space. 
	Since $\|.\|_{\mf m}\geq\|.\|_{\mc H}$, this implies that they are equivalent. 

	Trivially, $(ii)\Rightarrow(ii')$ holds. We will finish the proof by showing that $(i)\Leftrightarrow(iii)$, 
	$(iii')\Rightarrow(ii)$, and $(ii')\Rightarrow(iii')$. 
	To this end let us make a preliminary remark. Substituting the 
	explicit expression for $\nabla_{\!\mc H}$, cf.\ \cite[(2.5)]{baranov.woracek:spmaj}, we obtain 
	\begin{equation}\label{R19}
	\begin{gathered}
		\frac{\nabla_{\!\mc H}(z)}{\mf m_E(z)}=\frac{|z+i|}{|E(z)|}\Big(\frac{|E(z)|^2-|E(\qu z)|^2}{4\pi\Im z}\Big)^{\frac 12}=
			\frac{|z+i|}{2\sqrt{\pi\Im z}}\big(1-|\Theta(z)|^2\big)^{\frac 12}=
			\\
		=\frac 1{2\sqrt\pi}\frac{|z+i|}{\Im z}\big(1+|\Theta(z)|\big)^{\frac 12}\cdot
			\Big[\Im z\cdot(1-|\Theta(z)|)\Big]^{\frac 12},\quad z\in\bb C^+
			\,.
	\end{gathered}
	\end{equation}
	By \cite[(2.6)]{baranov.woracek:spmaj}, we have $\inf_{z\in\bb C^+}\mf m_E^{-1}\nabla_{\!\mc H}>0$. Since 
	the first factor of the last expression 
	in \eqref{R19} is bounded above and away from zero on $D=i[1,\infty)$, it follows that 
	(for the notations "$\lesssim$" and "$\asymp$" see \cite[\S3]{baranov.woracek:spmaj}) 
	\begin{equation}\label{R20}
		\nabla_{\!\mc H}|_D\lesssim\mf m\ \iff\ 
		\nabla_{\!\mc H}|_D\asymp\mf m\ \iff\ 
		\sup_{y\geq 1}\,y\big(1-|\Theta(iy)|\big)<\infty\,.
	\end{equation}

	\hspace*{0pt}\\[-2mm]\textit{$(i)\Leftrightarrow(iii)$:} 
	Let $C>0$ be such that $\|F\|_{\mf m}\leq C\|F\|_{\mc H}$. Then $|F(z)|\leq C\mf m(z)$, $z\in D$, $\|F\|_{\mc H}\leq 1$. 
	Thus $\nabla_{\!\mc H}|_D\lesssim\mf m$. Using 
	\eqref{R20} and \leref{R18}, we obtain that
	$R_{\mf m}(\mc H)=\mc H$ if and only if the condition $(iii)$ holds. 

	\hspace*{0pt}\\[-2mm]\textit{$(iii')\Rightarrow(ii)$:} 
	Let $\mc H(E)$ be such that $\qu{\dom\mc S_{\mc H}}\neq\mc H$, and assume without loss of generality 
	that $B\in\mc H(E)$. The function $-B^{-1}A$ has nonnegative 
	imaginary part throughout the upper half-plane, and its Herglotz integral representation is of the form 
	\[
		-\frac{A}{B}=pz+\sum_{B(t_n)=0}p_n\Big(\frac 1{t_n-z}-\frac{t_n}{1+t_n^2}\Big)
		\,,
	\]
	with some nonnegative numbers $p$ and $p_n$, $n\in\bb N$. 
	By \cite[Theorem 22]{debranges:1968} and its proof, the linear term $pz$ in this representation does not vanish. Hence, 
	for each $\alpha\in(0,\frac\pi2)$, 
	\[
		0<p=-\lim_{\substack{|z|\to\infty\\z\in\Gamma_\alpha}}
		\frac 1z\cdot\frac{A}{B}=
		i\lim_{\substack{|z|\to\infty\\z\in\Gamma_\alpha}}
		\frac 1z\cdot\frac{1+\Theta(z)}{1-\Theta(z)}
		\,.
	\]
	Since $1+\Theta(z)$ is bounded, it follows that $z(1-\Theta(z))$ is bounded throughout $\Gamma_\alpha$. 

	\hspace*{0pt}\\[-2mm]\textit{$(ii')\Rightarrow(iii')$:} 
	Again it is enough to consider the case that $\phi=0$. Hence, assume that 
	$\alpha\in(0,\frac\pi2)$, $C>0$, and a sequence $(z_n)_{n\in\bb N}$, $z_n\in\Gamma_\alpha$, is given such that 
	$|z_n|\cdot|1-\Theta(z_n)|\leq C$. Then, in particular, $\lim_{n\to\infty}\Theta(z_n)=1$, and it follows that 
	\[
		\liminf_{n\to\infty}\Big|\frac 1{z_n}\cdot\frac{A(z_n)}{B(z_n)}\Big|=
		\liminf_{n\to\infty}\Big|\frac 1{z_n}\cdot\frac{1+\Theta(z_n)}{1-\Theta(z_n)}\Big|\geq\frac 2C
		\,.
	\]
	If we had $\qu{\dom S_{\mc H}}=\mc H$, then by \cite[Theorem 29]{debranges:1968} and the proof of 
	\cite[Theorem 22]{debranges:1968} we would have 
	\[
		\lim_{\substack{|z|\to\infty\\z\in\Gamma_\alpha}}\frac 1z\cdot
		\frac{A(z)}{B(z)}=0
		\,,
	\]
	and obtain a contradiction. Thus $(iii')$ must hold. 
\end{proofof}

For the proof of \thref{R7}, we will employ the following two lemmata. 

\begin{lemma}{R12}
	Let $\mc H=\mc H(E)$ be a de Branges space, and let $(w_k)_{k\in\bb N}$, $w_k\in\bb C^+\cup\bb R$, 
	be a sequence of points with 
	\begin{equation}\label{R13}
		\lim_{k\to\infty}\frac{\mf m_E(w_k)}{\nabla_{\!\mc H}(w_k)}=0\,.
	\end{equation}
	Then there exists a subsequence $(w_{k(n)})_{n\in\bb N}$, such that the sequence 
	$(\tilde K(w_{k(n)},.))_{n\in\bb N}$ of normalized kernels 
	$\tilde K(w,.):=\nabla_{\!\mc H}(w)^{-1}K_{\mc H}(w,.)$ 
	is a Riesz sequence, i.e.\ a Riesz basis in its closed linear span. 
	
	In particular, if $(w_k)_{k\in\bb N}$ satisfies $|w_k|\to\infty$ and $\sup_{k\in\bb N}|\Theta(w_k)|<1$, 
	the hypothesis \eqref{R13} holds true. Here we have again set $\Theta:=E^{-1}E^\#$. 
\end{lemma}
\begin{proof}
	Let $(e_k)_{k\in\bb N}$ be a sequence of elements of some Hilbert space with $\|e_k\|=1$, $k\in\bb N$. Then, 
	in order that $(e_k)_{k\in\bb N}$ is a Riesz sequence, it is sufficient that 
	\[
		\sum_{\substack{n,m=1\\ n\neq m}}^\infty\big|(e_n,e_m)\big|^2<1
		\,,
	\]
	see e.g.\ \cite[VI.Theorem 2.1]{gohberg.krein:1965}. We compute 
	\[
		\big(\tilde K(w_m,.),\tilde K(w_n,.)\big)_{\mc H}=\frac 1{\nabla_{\!\mc H}(w_m)\nabla_{\!\mc H}(w_n)}
		\frac{E(w_n)\qu{E(w_m)}-E(\qu{w_m})E^\#(w_n)}{2\pi i(\qu{w_m}-w_n)}
		\,.
	\]
	Hence, by our assumption \eqref{R13}, for each fixed $m\in\bb N$ 
	\[
		\lim_{n\to\infty}\big(\tilde K(w_m,.),\tilde K(w_n,.)\big)_{\mc H}=0
		\,.
	\]
	Therefore we can extract a subsequence $(w_{k(n)})_{n\in\bb N}$ which satisfies 
	\[
		\sum_{\substack{n,m=1\\ n\neq m}}^\infty|\big(\tilde K(w_{k(n)},.),\tilde K(w_{k(m)},.)\big)_{\mc H}|^2\leq\frac 12
		\,, 
	\]
	and hence gives rise to a Riesz sequence $(\tilde K(w_{k(n)},.))_{n\in\bb N}$. 
	
	In order to see the last assertion, it is enough to consider the first line of \eqref{R19}: 
	\[
		\frac{\mf m_E(z)}{\nabla_{\!\mc H}(z)}=2\sqrt\pi\frac{(\Im z)^{\frac 12}}{|z+i|}\frac 1{(1-|\Theta(z)|^2)^{\frac 12}}
		\,.
	\]
	Hence, if $|w_n|\to\infty$ and $\sup_{n\in\bb N}|\Theta(w_n)|<1$, certainly \eqref{R13} will hold. 
\end{proof}

Throughout the following we will denote by $\mc U$ the linear space of all complex sequences $a=(a_k)_{k\in\bb N}$ with 
\[
	\|a\|_{\mc U}:=\sup_{n\in\bb N}\Big|\sum_{k=1}^n a_k\Big|<\infty
	\,.
\]
Then $\langle\mc U,\|.\|_{\mc U}\rangle$ is a Banach space. Actually, it is isometrically isomorphic to $\ell^\infty$ via 
the map $(a_n)_{n\in\bb N}\mapsto (\sum_{k=1}^na_k)_{n\in\bb N}$. 

\begin{lemma}{R10}
	Set $n_l:=l^3$, $l\in\bb N_0$, and consider the linear map $\Lambda$ which assigns to a sequence 
	$d=(d_k)_{k\in\bb N}$ the sequence $\Lambda(d)=(\Lambda(d)_k)_{k\in\bb N}$ defined as ($d_0:=0$) 
	\[
		\Lambda(d)_k:=\frac{d_l-d_{l-1}}{n_l-n_{l-1}},\qquad n_{l-1}<k\leq n_l,\ l\in\bb N
		\,.
	\]
	Then, whenever $d\in\ell^\infty$, we have $($with $C:=2\sqrt{1+\frac{\pi^2}{18}}$$)$ 
	\begin{equation}\label{R11}
		\|\Lambda(d)\|_2\leq C\|d\|_\infty,\quad\|\Lambda(d)\|_\infty\leq\|d\|_\infty,
		\quad \|\Lambda(d)\|_{\mc U}\leq\|d\|_\infty
		\,.
	\end{equation}
\end{lemma}
\begin{proof}
	We compute 
	\[
		\sum_{k=1}^\infty|\Lambda(d)_k|^2=\sum_{l=1}^\infty (n_l-n_{l-1})\Big|\frac{d_l-d_{l-1}}{n_l-n_{l-1}}\Big|^2\leq
		4\|d\|_\infty^2\sum_{l=1}^\infty\frac 1{l^3-(l-1)^3}=
	\]
	\[
		=4\|d\|_\infty^2\Big(1+\sum_{l=1}^\infty\frac 1{3l^2-3l+1}\Big)\leq 4\big(1+\frac{\pi^2}{18}\big)\,\|d\|_\infty^2
		\,,
	\]
	and this is the first inequality in \eqref{R11}. The second one is obvious since $\Lambda(d)_1=d_1$ and 
	$n_l-n_{l-1}\geq 2$, $l>1$. To see the last inequality, note that 
	\[
		\sum_{k=1}^{n_l}\Lambda(d)_k=d_l,\qquad l\in\bb N
		\,,
	\]
	and that the number $\sum_{k=1}^n\Lambda(d)_k$ lies on the line segment joining $d_{n_{l-1}}$ and $d_{n_l}$ when 
	$l\in\bb N$ is chosen such that $n_{l-1}\leq n\leq n_l$. 
	
\begin{center}
\setlength{\unitlength}{0.0002in}
\begingroup\makeatletter\ifx\SetFigFont\undefined%
\gdef\SetFigFont#1#2#3#4#5{%
  \reset@font\fontsize{#1}{#2pt}%
  \fontfamily{#3}\fontseries{#4}\fontshape{#5}%
  \selectfont}%
\fi\endgroup%
{\renewcommand{\dashlinestretch}{30}
\begin{picture}(10522,5355)(0,-10)

\path(12,2385)(10362,2385)
\blacken\path(10062.000,2310.000)(10362.000,2385.000)(10062.000,2460.000)(10062.000,2310.000)
\path(462,135)(462,4815)
\blacken\path(537.000,4515.000)(462.000,4815.000)(387.000,4515.000)(537.000,4515.000)

\dottedline{200}(912,4185)(1362,4635)(3612,585)
	(5862,3780)(8562,2835)(9822,3060)

\put(912,2360){\makebox(0,0)[c]{{\SetFigFont{6}{12.0}{\rmdefault}{\mddefault}{\updefault}$|$}}}
\put(1362,2360){\makebox(0,0)[c]{{\SetFigFont{6}{12.0}{\rmdefault}{\mddefault}{\updefault}$|$}}}
\put(3612,2360){\makebox(0,0)[c]{{\SetFigFont{6}{12.0}{\rmdefault}{\mddefault}{\updefault}$|$}}}
\put(5817,2360){\makebox(0,0)[c]{{\SetFigFont{6}{12.0}{\rmdefault}{\mddefault}{\updefault}$|$}}}
\put(8562,2360){\makebox(0,0)[c]{{\SetFigFont{6}{12.0}{\rmdefault}{\mddefault}{\updefault}$|$}}}

\put(912,1800){\makebox(0,0)[c]{{\SetFigFont{6}{12.0}{\rmdefault}{\mddefault}{\updefault}$n_{\!1}$}}}
\put(1362,1800){\makebox(0,0)[c]{{\SetFigFont{6}{12.0}{\rmdefault}{\mddefault}{\updefault}$\,n_{\!2}$}}}
\put(3612,1800){\makebox(0,0)[c]{{\SetFigFont{6}{12.0}{\rmdefault}{\mddefault}{\updefault}$n_{\!3}$}}}
\put(5817,1800){\makebox(0,0)[c]{{\SetFigFont{6}{12.0}{\rmdefault}{\mddefault}{\updefault}$n_{\!4}$}}}
\put(8562,1800){\makebox(0,0)[c]{{\SetFigFont{6}{12.0}{\rmdefault}{\mddefault}{\updefault}$n_{\!5}$}}}
\put(10407,1700){\makebox(0,0)[c]{{\SetFigFont{8}{12.0}{\rmdefault}{\mddefault}{\updefault}$n$}}}

\put(912,4185){\makebox(0,0)[c]{{\SetFigFont{6}{12.0}{\rmdefault}{\mddefault}{\updefault}$\bullet$}}}
\put(1362,4635){\makebox(0,0)[c]{{\SetFigFont{6}{12.0}{\rmdefault}{\mddefault}{\updefault}$\bullet$}}}
\put(3612,585){\makebox(0,0)[c]{{\SetFigFont{6}{12.0}{\rmdefault}{\mddefault}{\updefault}$\bullet$}}}
\put(5817,3780){\makebox(0,0)[c]{{\SetFigFont{6}{12.0}{\rmdefault}{\mddefault}{\updefault}$\bullet$}}}
\put(8562,2835){\makebox(0,0)[c]{{\SetFigFont{6}{12.0}{\rmdefault}{\mddefault}{\updefault}$\bullet$}}}

\put(822,4905){\makebox(0,0)[c]{{\SetFigFont{6}{12.0}{\rmdefault}{\mddefault}{\updefault}$d_{\!1}$}}}
\put(1362,5220){\makebox(0,0)[c]{{\SetFigFont{6}{12.0}{\rmdefault}{\mddefault}{\updefault}$d_{\!2}$}}}
\put(3612,45){\makebox(0,0)[c]{{\SetFigFont{6}{12.0}{\rmdefault}{\mddefault}{\updefault}$d_3$}}}
\put(5862,4185){\makebox(0,0)[c]{{\SetFigFont{6}{12.0}{\rmdefault}{\mddefault}{\updefault}$d_4$}}}
\put(8562,3300){\makebox(0,0)[c]{{\SetFigFont{6}{12.0}{\rmdefault}{\mddefault}{\updefault}$d_5$}}}

\put(10500,4700){\makebox(0,0)[c]{{\SetFigFont{8}{12.0}{\rmdefault}{\mddefault}{\updefault}$\sum\limits_{k=1}^n\Lambda(d)_k$}}}

\end{picture}
}
\end{center}
\hspace*{0pt}\vspace*{-8mm}
\end{proof}

\begin{proofof}{\thref{R7}}
	The fact that $(i)\Rightarrow(ii)$ is clear. Let us assume that $(ii)$ holds. 

\hspace*{0pt}\\[-2mm]\textit{Step 1, Extracting a sparse sequence $(v_k)_{k\in\bb N}$:} 
	Let $\delta>0$ be chosen according to \eqref{R8}, and let 
	$(w_k)_{k\in\bb N}$ be a sequence of points with $\Im w_k\geq\delta$, $|w_k|\geq 1$, 
	$|w_k|<|w_{k+1}|$, $|w_k|\to\infty$, and $\lim_{k\to\infty}|\Theta(w_k)|=0$. 
	By \leref{R12}, we may extract a subsequence $(w_{k(n)})_{n\in\bb N}$ such that the normalized kernel 
	functions $\tilde K(w_{k(n)},.)$, $n\in\bb N$, form a Riesz sequence in $\mc H$. Note that therefore also 
	each subsequence of this sequence of functions is a Riesz sequence. Since $\lim_{k\to\infty}|\Theta(w_k)|=0$, 
	we may extract a subsequence $(w_{k(n(l))})_{l\in\bb N}$ such that $\sum_{l=1}^\infty|\Theta(w_{k(n(l))})|\leq\frac 14$. 
	We thus have found a sequence $(u_k)_{k\in\bb N}$ with the following properties 
	\begin{enumerate}[$(i)$]
	\item $\Im u_k\geq\delta>0$, $|u_k|\geq 1$, $|u_k|<|u_{k+1}|$, $k\in\bb N$, $|u_k|\to\infty$;
	\item The sequence $(\tilde K(u_k,.))_{k\in\bb N}$ is a Riesz sequence in $\mc H$;
	\item $\sum_{n=1}^\infty|\Theta(u_n)|\leq\frac 14$.
	\end{enumerate}
	A straightforward induction shows that we may extract yet another subsequence $(v_k)_{k\in\bb N}$ 
	of $(u_k)_{k\in\bb N}$ which satisfies 
	\begin{enumerate}[$(i)$]
	\setcounter{enumi}{3}
	\item $\sum_{n=1}^k|v_n|\leq\frac 18 |v_kv_{k+1}|^{\frac 12}$, $k\in\bb N$;
	\item $\sum_{n=l}^k\frac 1{|v_n|}<\frac 18 |v_{l-1}v_l|^{-\frac 12}$, $l=2,\ldots,k$; and hence, letting 
		$k\to\infty$, 
		\[
			\sum_{n=l}^\infty\frac 1{|v_n|}\leq\frac 18 |v_{l-1}v_l|^{-\frac 12},\quad l\geq 2
			\,. 
		\]
	\end{enumerate}
	Since each of the properties $(i)$--$(iii)$ remains valid when passing to subsequences, the sequence $(v_k)_{k\in\bb N}$ 
	satisfies $(i)$--$(v)$. 

\hspace*{0pt}\\[-2mm]\textit{Step 2, Definition of $\Psi$:} 
	Since $(\tilde K(v_k,.))_{k\in\bb N}$ is a Riesz sequence in $\mc H$, the map $\rho$ defined as 
	\begin{equation}\label{R15}
		\rho:(a_k)_{k\in\bb N}\mapsto\sum_{k=1}^\infty a_k\tilde K(v_k,.)
	\end{equation}
	induces a bicontinuous embedding of $\ell^2$ into $\mc H$. Set 
	\[
		\mu_k:=i\bigg(\pi \frac{1-|\Theta(v_k)|^2}{\Im v_k}\bigg)^{1/2}\,
		\frac{|E(v_k)|}{\qu{E(v_k)}},\qquad k\in\bb N
		\,.
	\]
	Since $\Im v_k\geq\delta>0$, $k\in\bb N$, the sequence $(\mu_k)_{k\in\bb N}$ belongs to $\ell^\infty$. Thus the 
	multiplication operator 
	\begin{equation}\label{R22}
		\mu:(a_k)_{k\in\bb N}\mapsto(\mu_ka_k)_{k\in\bb N}
	\end{equation}
	induces a bounded operator of $\ell^2$ into itself. Finally, let $\Lambda$ be the operator 
	constructed in \leref{R10} considered as an element of $\mc B(\ell^\infty,\ell^2)$, and set 
	\[
		\Psi:=\rho\circ\mu\circ\Lambda
		\,.
	\]
	Then $\Psi\in\mc B(\ell^\infty,\mc H)$. 

\hspace*{0pt}\\[-2mm]\textit{Step 3, $\Psi\in\mc B(\ell^\infty,R_{\mf m}(\mc H))$:} 
	Let $d\in\ell^\infty$. We have, by our choice of $\mu_k$, 
	\begin{equation}\label{R23}
	\begin{split}
		\Psi(d)(z)=&\sum_{n=1}^\infty i\bigg( \pi\frac{1-|\Theta(v_n)|^2}
                        {\Im v_n}\bigg)^{1/2}  \frac{|E(v_n)|}{\qu{E(v_n)}}
			\Lambda(d)_n\cdot\tilde K(v_n,z)=
			\\
		=&\sum_{n=1}^\infty i\bigg(\pi\frac{1\!-\!|\Theta(v_n)|^2}{\Im v_n}\bigg)^{1/2}\,
			\frac{|E(v_n)|}{\qu{E(v_n)}}\Lambda(d)_n
			\cdot\!\bigg[\Big(\frac{|E(v_n)|^2-|E(\qu{v_n})|^2}{4\pi\Im v_n}\Big)^{\!-1/2}\cdot
			\\
		&\cdot\frac{E(z)\qu{E(v_n)}-E(\qu{v_n})E^\#(z)}{2\pi i(\qu{v_n}-z)}\bigg]=
			E(z)\sum_{n=1}^\infty\Lambda(d)_n\frac{1-\qu{\Theta(v_n)}\Theta(z)}{\qu{v_n}-z}
			\,,
	\end{split}
	\end{equation}
	and hence 
	\[
		\frac{\Psi(d)(z)}{E(z)}=\sum_{n=1}^\infty\Lambda(d)_n\frac{1-\qu{\Theta(v_n)}\Theta(z)}{\qu{v_n}-z}=
		\sum_{n=1}^\infty\Big[\frac{\Lambda(d)_n}{\qu{v_n}-z}-\Lambda(d)_n\qu{\Theta(v_n)}\frac{\Theta(z)}{\qu{v_n}-z}\Big]
		\,.
	\]
	Let $y\in[|v_1|,\infty)$ be given, and choose $k\in\bb N$ such that $|v_k|\leq y\leq|v_{k+1}|$. Using 
	\eqref{R11} and the properties $(iii)$, $(iv)$, $(v)$ of $(v_k)_{k\in\bb N}$, 
	we obtain the below estimates. Note thereby that $v_n\in\bb C^+$ and hence 
	$|\qu{v_n}-iy|\geq\max\{|v_n|,y\}$. 
	\begin{multline*}
		\Big|\sum_{n=1}^{k-1}\frac{\Lambda(d)_n}{\qu{v_n}-iy}\Big|=
		\Big|\sum_{n=1}^{k-1}\frac{\Lambda(d)_n}{iy}\Big(-1+\frac{\qu{v_n}}{\qu{v_n}-iy}\Big)\Big|\leq
		\\
		\shoveright{
		\leq\frac 1y\|\Lambda(d)\|_{\mc U}+\frac{\|\Lambda(d)\|_\infty}{y^2}
		\hspace*{-4mm}\underbrace{\ \sum_{n=1}^{k-1}|v_n|\ }_{\leq\frac 18|v_{k-1}v_k|^{\frac 12}\leq\frac y8}\hspace*{-4mm}
		\leq\frac{9}{8y}\|d\|_\infty
		\,,}\\[-12mm]
	\end{multline*}
	\begin{multline*}
		\Big|\frac{\Lambda(d)_k}{\qu{v_k}-iy}+\frac{\Lambda(d)_{k+1}}{\qu{v_{k+1}}-iy}\Big|
		\leq\frac 2y\|\Lambda(d)\|_\infty\leq\frac 2y\|d\|_\infty
		\,,\\[-10mm]
	\end{multline*}
	\begin{multline*}
		\Big|\sum_{n=k+2}^\infty\frac{\Lambda(d)_n}{\qu{v_n}-iy}\Big|\leq\|\Lambda(d)\|_\infty
		\hspace*{-5mm}\underbrace{\ \sum_{n=k+2}^\infty\frac 1{|v_n|}\ }_{
			\leq\frac 18|v_{k+1}v_{k+2}|^{-\frac 12}\leq\frac 1{8y}}\hspace*{-5mm}
		\leq\frac 1{8y}\|d\|_\infty
		\,,\\[-10mm]
	\end{multline*}
	\begin{multline*}
		\Big|\sum_{n=1}^\infty\Lambda(d)_n\qu{\Theta(v_n)}\frac{\Theta(iy)}{\qu{v_n}-iy}\Big|\leq\|\Lambda(d)\|_\infty
		\frac 1y\sum_{n=1}^\infty|\Theta(v_n)|\leq\frac 1{4y}\|d\|_\infty
		\,.\\[-5mm]
	\end{multline*}
	Altogether, it follows that 
	\[
		y\Big|\frac{\Psi(d)(iy)}{E(iy)}\Big|\leq\frac 72\|d\|_\infty,\qquad y\in[|v_1|,\infty)
		\,.
	\]
	Consider the majorant $\tilde{\mf m}(z):=\frac{|E(z)|}{|z|}$ defined on the ray $\tilde D:=i[|v_1|,\infty)$. 
	The above estimate, together with the already known fact that $\Psi\in\mc B(\ell^\infty,\mc H)$, implies that 
	$\Psi(\ell^\infty)\subseteq R_{\tilde{\mf m}}(\mc H)$ and 
	$\Psi\in\mc B(\ell^\infty,\langle R_{\tilde{\mf m}}(\mc H),\|.\|_{\tilde{\mf m}}\rangle)$. 

	Since the domains of $\mf m$ and $\tilde{\mf m}$ differ only by a bounded set, and $\tilde{\mf m}\asymp\mf m|_{\tilde D}$, 
	we have $R_{\mf m}(\mc H)=R_{\tilde{\mf m}}(\mc H)$ and the norms $\|.\|_{\mf m}$ and $\|.\|_{\tilde{\mf m}}$ are 
	equivalent. We conclude that 
	\[
		\Psi\in\mc B\big(\ell^\infty,\langle R_{\mf m}(\mc H),\|.\|_{\mf m}\rangle\big)
		\,. 
	\]

\hspace*{0pt}\\[-2mm]\textit{Step 4, $\Psi$ is bicontinuous:} 
	Let $n_l=l^3$, cf.\ \leref{R10}, and consider the values $y_k:=|v_{n_k}v_{n_k+1}|^{\frac 12}$, $k\in\bb N$. 
	Then, for each sequence $d\in\ell^\infty$, 
	\begin{multline*}
		\frac{\Psi(d)(iy_k)}{E(iy_k)}=-\frac 1{iy_k}\underbrace{\sum_{n=1}^{n_k}\Lambda(d)_n}_{=d_k}+
			\sum_{n=1}^{n_k}\Lambda(d)_n\frac 1{iy_k}\frac{\qu{v_n}}{\qu{v_n}-iy_k}+
			\\
		\shoveright{
			+\sum_{n=n_k+1}^\infty\frac{\Lambda(d)_n}{\qu{v_n}-iy_k}-
			\sum_{n=1}^\infty\Lambda(d)_n\qu{\Theta(v_n)}\frac{\Theta(iy_k)}{\qu{v_n}-iy_k}
			\,.}\\
	\end{multline*}
	However, we estimate 
	\begin{multline*}
		\Big|\sum_{n=1}^{n_k}\Lambda(d)_n\frac 1{iy_k}\frac{\qu{v_n}}{\qu{v_n}-iy_k}\Big|\leq
		\frac{\|\Lambda(d)\|_\infty}{y_k^2}
		\hspace*{-5mm}\underbrace{\ \sum_{n=1}^{n_k}|v_n|\ }_{\leq\frac 18|v_{n_k}v_{n_k+1}|^{\frac 12}=
		\frac{y_k}8}\hspace*{-6mm}\leq\frac 1{8y_k}\|d\|_\infty
		\,,\\[-10mm]
	\end{multline*}
	\begin{multline*}
		\Big|\sum_{n=n_k+1}^\infty\frac{\Lambda(d)_n}{\qu{v_n}-iy}\Big|\leq\|\Lambda(d)\|_\infty
		\hspace*{-5mm}\underbrace{\ \sum_{n=n_k+1}^\infty\frac 1{|v_n|}\ }_{
			\leq\frac 18|v_{n_k}v_{n_k+1}|^{-\frac 12}=\frac 1{8y_k}}\hspace*{-5mm}
		\leq\frac 1{8y_k}\|d\|_\infty
		\,,\\[-10mm]
	\end{multline*}
	\begin{multline*}
		\Big|\sum_{n=1}^\infty\Lambda(d)_n\qu{\Theta(v_n)}\frac{\Theta(z)}{\qu{v_n}-z}\Big|\leq\|\Lambda(d)\|_\infty
		\frac 1{y_k}\sum_{n=1}^\infty|\Theta(v_n)|\leq\frac 1{4y_k}\|d\|_\infty
		\,,\\[-5mm]
	\end{multline*}
	and it follows that 
	\[
		\|\Psi(d)\|_{\tilde{\mf m}}\geq\frac{\Psi(d)(iy_k)}{\tilde{\mf m}(iy_k)}\geq
		|d_k|-\frac 12\|d\|_\infty
		\,.
	\]
	Taking the supremum over $k\in\bb N$, yields $\|\Psi(d)\|_{\tilde{\mf m}}\geq\frac 12\|d\|_\infty$. 
\end{proofof}

With the help of \leref{R12} we can also show that $R_{\mf m}(\mc H)$ may contain 
infinite dimensional closed subspaces $\mc M$ on which the norms $\|.\|_{\mf m}$ and $\|.\|_{\mc H}$ 
are equivalent, even if $R_{\mf m}(\mc H)$ itself is not $\|.\|_{\mc H}$-closed. 

\begin{proposition}{R14}
	Let $\mc H=\mc H(E)$ be a de~Branges space, and set $\mf m:=\mf m_{E}|_{i[1,\infty)}$. 
	Moreover, denote $\Theta:=E^{-1}E^\#$. Assume that there exists a sequence 
	$(w_k)_{k\in\bb N}$, $w_k\in\bb C^+$, such that 
	\[
		\liminf_{k\in\bb N}\Im w_k=0\ \text{ and }\ \sup_{k\in\bb N}|\Theta(w_k)|<1
		\,.
	\]
	Then there exists an infinite dimensional subspace $\mc M$ of 
	$R_{\mf m}(\mc H)$ which is $\|.\|_{\mc H}$-closed. In particular, the norms $\|.\|_{\mf m}|_{\mc M}$ and 
	$\|.\|_{\mc H}|_{\mc M}$ are equivalent. 
\end{proposition}
\begin{proof}
	The normalized kernel function $\tilde K(w,.)$, $w\in\bb C^+$, is explicitly given as 
	\[
		\tilde K(w,z)=\Big(\frac{|E(w)|^2-|E(\qu w)|^2}{4\pi\Im w}\Big)^{-\frac 12}
		\frac{E(z)\qu{E(w)}-E(\qu w)E^\#(z)}{2\pi i(\qu w-z)}=
	\]
	\[
		=E(z)\frac{1-\qu{\Theta(w)}\Theta(z)}{i\sqrt\pi(\qu w-z)}
		\Big(\frac{\Im w}{1-|\Theta(w)|^2}\Big)^{\frac 12}\frac{\qu{E(w)}}{|E(w)|}
		\,.
	\]
	We obtain the estimate ($\delta:=\sup_{n\in\bb N}|\Theta(w_n)|$) 
	\[
		\Big|\frac{\tilde K(w_n,iy)}{\mf m(iy)}\Big|\leq\frac{2}{\sqrt\pi}\frac{y+1}y
		\frac{\sqrt{\Im w_n}}{1-\delta^2},\quad y\geq 1
		\,.
	\]
	Hence $\tilde K(w_n,.)\in R_{\mf m}(\mc H)$ and 
	\[
		\Big\|\frac{\tilde K(w_n,.)}{\mf m}\Big\|_\infty\leq \frac{4\sqrt{\Im w_n}}{\sqrt\pi (1-\delta^2)}
		\,.
	\]
	Since $\liminf_{n\to\infty}\Im w_n=0$, we can extract 
	a subsequence $(v_n)_{n\in\bb N}$ of $(w_n)_{n\in\bb N}$ such that $\sum_{n\in\bb N}\Im v_n<\infty$. 
	Moreover, by \leref{R12}, the choice of $(v_n)_{n\in\bb N}$ can be made such that 
	$(\tilde K(v_k,.))_{k\in\bb N}$, is a Riesz sequence. Consider the bicontinuous embedding 
	$\rho:\ell^2\to\mc H$ defined as in \eqref{R15}. Since,	the sequence 
	$(\|\mf m^{-1}\tilde K(v_n,.)\|_\infty)_{n\in\bb N}$ is by our choice of $(v_n)_{n\in\bb N}$ square summable, 
	for each sequence $a=(a_n)_{n\in\bb N}\in\ell^2$ 
	\[
		\Big\|\frac 1{\mf m}\sum_{n\in\bb N}a_n\tilde K_n\Big\|_\infty\leq
		\sum_{n\in\bb N}|a_n|\cdot\Big\|\frac{\tilde K_n}{\mf m}\Big\|_\infty\leq 
		\|a\|_2\cdot\big\|(\|\mf m^{-1}\tilde K_n\|_\infty)_{n\in\bb N}\big\|_2
		<\infty
		\,.
	\]
	Thus every element of $\mc M:=\clos_{\mc H}\spn\{\tilde K(v_n,.):\,n\in\bb N\}$ belongs to 
	$R_{\mf m}(\mc H)$. Since $\|.\|_{\mf m}\geq\|.\|_{\mc H}$, this implies that $\|.\|_{\mf m}|_{\mc M}$ and 
	$\|.\|_{\mc H}|_{\mc M}$ are equivalent. 
\end{proof}

Let us explicitly state the following observation, which we obtain from \prref{R14} in conjunction with \thref{R7}. 

\begin{corollary}{R16}
	Let $\mc H=\mc H(E)$ be a de~Branges space, and set $\mf m:=\mf m_{E}|_{i[1,\infty)}$. 
	Denote by $(w_k)_{k\in\bb N}$ the sequence of zeros of $E^\#$, and assume that 
	\[
		\liminf_{k\to\infty}\Im w_k=0,\quad \limsup_{k\to\infty}\Im w_k>0
		\,.
	\]
	Then the space $\langle R_{\mf m}(\mc H),\|.\|_{\mf m}\rangle$ 
	contains two closed infinite dimensional subspaces $\mc M_1$, $\mc M_2$, such that $\|.\|_{\mf m}|_{\mc M_1}$ 
	is equivalent to the Hilbert space norm $\|.\|_{\mc H}|_{\mc M_1}$, and 
	$\langle\mc M_2,\|.\|_{\mf m}|_{\mc M_2}\rangle$ is bicontinuously isomorphic to $\langle\ell^\infty,\|.\|_\infty\rangle$. 
	\popQED{\qed}
\end{corollary}

%
%
\begin{flushleft}
	\textbf{b. Majorization along the real line.}
\end{flushleft}
\vspace*{-2mm}
In this subsection we study majorization on $D:=\bb R$. It turns out that in this case the situation is different; 
the geometry of $R_{\mf m_{E_1}|_D}(\mc H)$ is always complicated. 

\begin{theorem}{R21}
	Let $\mc H$ be a de~Branges space, let $\mc L=\mc H(E_1)\in\Sub^*\mc H$, and set $\mf m:=\mf m_{E_1}|_{\bb R}$. 
	Then there exists a bicontinuous embedding of $\langle\ell^\infty,\|.\|_\infty\rangle$ into 
	$\langle R_{\mf m}(\mc H)\cap\mc L,\|.\|_{\mf m}\rangle$. In particular, $R_{\mf m}(\mc H)$ is neither 
	separable nor reflexive. 
\end{theorem}

The basic idea for the proof of this result is similar to the one which led to \thref{R7}, still there occur differences 
in the line of argumentation (see also \reref{R26}). We have to consider separately the cases whether the zeros $w_k$ of 
$\Theta:=E_1^{-1}E^\#$ do approach the real axis or are contained in some Stolz angle. 
\\

\begin{proofof}{\thref{R21}. Case 1: zeros approaching the real axis.}
	Assume that
	\[
		\liminf\limits_{k\to\infty,\Re w_k>0}\!\!\frac{\Im w_k}{|w_k|}=0.
        \]
	Let $(w_{k(n)})_{n\in\bb N}$ be a subsequence of $(w_k)_{k\in\bb N}$ with 
	$\Re w_{k(n)}>0$ and $\lim_{n\to\infty}|w_n|^{-1}\Im w_n=0$. 
	From this sequence we can extract yet another subsequence $(v_k)_{k\in\bb N}$ which satisfies 
	\begin{enumerate}[$(i)$]
	\item $|v_k+i|\leq 2\Re v_k$, $k\in\bb N$,
	\item $\Re v_{k+1}>2\Re v_k$, $k\in\bb N$, 
	\item $\sum_{n=1}^\infty\frac{\Im v_n}{|v_n+i|}\leq\frac 1{24}$. 
	\end{enumerate}
	Denote by $K_1(w,z)$ the reproducing kernel of $\mc L$, and set $\tilde K_1(w,.):=\|K_1(w,.)\|_{\mc H}^{-1}K_1(w,.)$. 
	The functions $\tilde K_1(v_k,.)$ are normalized and belong to the subspace $\mc L$. Hence 
	the map $\rho$ defined by \eqref{R15}, using $\tilde K_1(v_k,.)$ in place of $\tilde K(v_k,.)$, 
	maps the space $\ell^1$ contractively into $\mc L$. Set 
	\[
		\gamma_k:=\frac{\Im v_k}{|v_k+i|},\qquad k\in\bb N
		\,. 
	\]
	Then, by $(iii)$, the multiplication operator $\gamma:(a_k)_{k\in\bb N}\mapsto(\gamma_ka_k)_{k\in\bb N}$ maps 
	$\ell^\infty$ boundedly into $\ell^1$. Finally, the multiplication operator $\mu$ defined in \eqref{R22} maps 
	$\ell^\infty$ boundedly into itself. Altogether, we have 
	\[
		\Psi_1:=\rho\circ\gamma\circ\mu\in\mc B(\ell^\infty,\mc H)
		\,.
	\]
	Note that, similar as in \eqref{R23}, 
	\[
		\frac{\Psi_1(c)(z)}{E_1(z)}=\sum_{n=1}^\infty\frac{c_n\gamma_n}{\qu{v_n}-z},\qquad 
		c=(c_k)_{k\in\bb N}\in\ell^\infty
		\,.
	\]
	In order to estimate $\mf m$-norms let some point 
	$x\in\bb R\setminus(-\Re v_1,\Re v_1)$ be given. If $x\leq-\Re v_1\leq 0$, clearly 
	$|E_1^{-1}(x)\Psi_1(c)(x)|\leq\frac{\|c\|_\infty}{|x|}\sum_{n=1}^\infty\gamma_n\leq\frac{\|c\|_\infty}{24|x|}$. 
	Next, note that 
	\[
		\frac 1{|\qu{v_n}-x|}\leq
		\begin{cases}
			\frac 1{\Re v_n-x}\leq\frac 1x, &\ x\leq\frac 12\Re v_n,\\
			\frac 1{\Re v_k-\Re v_{k-1}}\leq\frac 2{\Re v_k}\leq\frac 4x, &\ n\leq k-1,\ \Re v_k\leq x\leq 2\Re v_k,\\
			\frac 1{x-\Re v_n}\leq\frac 2x, &\ 2\Re v_n\leq x,
		\end{cases}
	\]
	\[
		\frac 1{|v_n+i|}\leq\frac 1{\Re v_n}\leq
		\begin{cases}
			\frac 1x, &\ x\leq\Re v_n,\\
			\frac 2x, &\ x\leq 2\Re v_n.
		\end{cases}
	\]
	Using these inequalities, it follows that 
	\begin{multline*}
		\Big|\sum_{n=1}^\infty\frac{\gamma_n}{\qu{v_n}-x}\Big| \leq
		\sum_{n=1}^{k-1}\frac{\gamma_n}{|\qu{v_n}-x|}+
		\frac 1{|\qu{v_k}-x|}\frac{\Im v_k}{|v_k+i|}+\frac 1{|\qu{v_{k+1}}-x|}\frac{\Im v_{k+1}}{|v_{k+1}+i|} 
			\\
		+ \sum_{n=k+2}^\infty\frac{\gamma_n}{|\qu{v_n}-x|}\leq
		\frac 1{x}\Big(\frac 4{24}+2+1+\frac 1{24}\Big),\qquad \Re v_k\leq x\leq\Re v_{k+1}
		\,.
	\end{multline*}
	It follows that $\Psi_1\in\mc B(\ell^\infty,R_{\tilde{\mf m}}(\mc H))$ where $\tilde{\mf m}$ is the 
	majorant $\tilde{\mf m}(x):=|x|^{-1}|E(x)|$ defined on $\tilde D:=\bb R\setminus(-\Re v_1,\Re v_1)$. 
	
	We can also obtain an estimate from below. Set $x_k:=\Re v_k$, $k\in\bb N$, then 
	\[
		\Big|\frac{\Psi_1(c)(x_k)}{E_1(x_k)}\Big|\geq\frac{|c_k|}{|\qu{v_k}-x_k|}\frac{\Im v_k}{|v_k+i|}-
		\|c\|_\infty\sum_{\substack{n=1\\ n\neq k}}^\infty\frac 1{|\qu{v_n}-x|}\frac{\Im v_n}{|v_n+i|}\geq 
	\]
	\[
		\geq\frac{|c_k|}{2\Re v_k}-\frac{\|c\|_\infty}{x_k}\Big(\frac 4{24}+
		\underbrace{\frac{\Im v_{k+1}}{|v_{k+1}+i|}}_{\leq\frac 1{24}}+\frac 1{24}\Big)\geq
		\frac 1{x_k}\Big(\frac{|c_k|}2-\frac{\|c\|_\infty}4\Big)
		\,.
	\]
	Taking the supremum over all $k\in\bb N$ yields $\|\Psi_1(c)\|_{\tilde{\mf m}}\geq\frac 14\|c\|_\infty$. 
	Since $R_{\tilde{\mf m}}(\mc H)=R_{\mf m}(\mc H)$ and the norms $\|.\|_{\tilde{\mf m}}$ and $\|.\|_{\mf m}$ are equivalent, 
	we obtain that $\Psi_1$ maps $\ell^\infty$ bicontinuously into $\langle R_{\mf m}(\mc H)\cap\mc L,\|.\|_{\mf m}\rangle$. 
\end{proofof}

The case when $\liminf\limits_{k\to\infty,\Re w_k<0}\!\!\frac{\Im w_k}{|w_k|}=0$ is treated completely similar, and we will 
therefore omit explicit proof. 
\\

\begin{proofof}{\thref{R21}. Case 2: Zeros in a Stolz angle.}
	Assume that there exists $\alpha\in \big(0,\frac\pi 2\big)$ such that
	$w_n\in\Gamma_\alpha$, $n\in\bb N$.
	We can extract a subsequence $(v_k)_{k\in\bb N}$ of $(w_k)_{k\in\bb N}$ with 
	\begin{enumerate}[$(i)$]
	\item $\Im v_k\geq 1$, $\Im v_k<\Im v_{k+1}$, $k\in\bb N$, $\Im v_k\to\infty$, 
	\item $\sum_{n=1}^\infty(\Im v_k)^{-\frac 12}\leq 1$. 
	\end{enumerate}
	As seen by a straightforward induction, we may choose $(v_k)_{k\in\bb N}$ in such a way that moreover 
	\begin{enumerate}[$(i)$]
	\setcounter{enumi}{2}
	\item $\sum_{n=1}^k|v_n|\leq\frac{\sin\alpha}8(\Im v_k\Im v_{k+1})^{\frac 12}$,
	\item $\sum_{n=l}^k\frac 1{\Im v_k}<\frac 18(\Im v_{l-1}\Im v_l)^{-\frac 12}$, $l=2,\ldots,k$; and hence, letting 
		$k\to\infty$, 
		\[
			\sum_{n=l}^\infty\frac 1{\Im v_k}\leq\frac 18(\Im v_{l-1}\Im v_l)^{-\frac 12},\quad l\geq 2
			\,. 
		\]
	\end{enumerate}
	Again the map $\rho$ defined by \eqref{R15} with $\tilde K_1(v_k,.)$ 
	is a contraction of $\ell^1$ into $\mc L$. Let $\mu$ be the 
	multiplication operator defined in \eqref{R22}. By $(ii)$ we have $(\mu_k)_{k\in\bb N}\in\ell^1$, and hence 
	$\mu\in\mc B(\ell^\infty,\ell^1)$. Finally, let $\Lambda$ be the map defined in \leref{R10}, then  
	$\Lambda$ maps $\ell^\infty$ contractively into itself. Consider the map 
	\[
		\Psi:=\rho\circ\mu\circ\Lambda\in\mc B(\ell^\infty,\mc L)
		\,.
	\]
	We have to estimate $\mf m$-norms. Note that always $|\qu{v_n}-x|\geq\max\{|x|\sin\alpha,\Im v_n\}$. 
	Let $x\in\bb R\setminus(-\Im v_1,\Im v_1)$ be given, and choose $k\in\bb N$ such that 
	$\Im v_k\leq|x|\leq\Im v_{k+1}$. Then 
	\begin{multline*}
		\Big|\sum_{n=1}^{k-1}\frac{\Lambda(d)_n}{\qu{v_n}-x}\Big|=
			\Big|\sum_{n=1}^{k-1}\frac{\Lambda(d)_n}x\Big(-1+\frac{\qu{v_n}}{\qu{v_n}-x}\Big)\Big|\leq
			\\
		\shoveright{\leq\frac 1{|x|}\|\Lambda(d)\|_{\mc U}+\frac{\|\Lambda(d)\|_\infty}{|x|^2\sin\alpha}
			\hspace*{-12mm}\underbrace{\ \sum_{n=1}^{k-1}|v_n|}_{\leq\frac{\sin\alpha}8(\Im v_{k-1}\Im v_k)^{\frac 12}
			\leq\frac{|x|\sin\alpha}8}\hspace*{-12mm}\leq\frac 9{8|x|}\|d\|_\infty,
			}\\[-12mm]
	\end{multline*}
	\begin{multline*}
		\Big|\frac{\Lambda(d)_k}{\qu{v_k}-x}+\frac{\Lambda(d)_{k+1}}{\qu{v_{k+1}}-x}
		\Big|\leq\frac{2\|\Lambda(d)\|_\infty}{|x|\sin\alpha}\leq
		\frac{2}{|x|\sin\alpha}\|d\|_\infty,\\[-10mm]
	\end{multline*}
	\begin{multline*}
		\Big|\sum_{n=k+2}^\infty\frac{\Lambda(d)_n}{\qu{v_n}-x}\Big|\leq\|\Lambda(d)\|_\infty
		\hspace*{-9mm}\underbrace{\ \sum_{n=k+2}^\infty\frac 1{\Im v_n}}_{\leq\frac 18(\Im v_{k+1}
		\Im v_{k+2})^{-\frac 12}\leq\frac 1{8|x|}}\hspace*{-8mm}\leq\frac 1{8|x|}\|d\|_\infty.\\[-7mm]
	\end{multline*}
	Putting together these inequalities, it follows that 
	\begin{equation}\label{R24}
		\Big|\frac{\Psi(d)(x)}{E(x)}\Big|\leq\Big(\frac 54+\frac 2{\sin\alpha}\Big)\frac 1{|x|}\|d\|_\infty,\quad 
		|x|\geq\Im v_1
		\,.
	\end{equation}
	Consider the majorant $\tilde{\mf m}(x):=|x|^{-1}|E(x)|$ defined on $\tilde D:=\bb R\setminus(-\Im v_1,\Im v_1)$. 
	Then \eqref{R24} says that $\Psi\in\mc B(\ell^\infty,\langle R_{\tilde{\mf m}}(\mc H)\cap\mc L,\|.\|_{\tilde{\mf m}}\rangle)$. 
	Equivalence of norms implies 
	\[
		\Psi\in\mc B(\ell^\infty,\langle R_{\mf m}(\mc H)\cap\mc L,\|.\|_{\mf m}\rangle)
		\,.
	\]
	Let $n_l:=l^3$ as in \leref{R10}, and set $x_k:=(\Im v_{n_k}\Im v_{n_k+1})^{\frac 12}$. Then we can estimate 
	\begin{multline*}
		\Big|\sum_{n=1}^{n_k}\frac 1{x_k}\frac{\Lambda(d)_n\qu{v_n}}{\qu{v_n}-x_k}\Big|\leq
			\frac{\|\Lambda(d)\|_\infty}{|x_k|^2\sin\alpha}\sum_{n=1}^{n_k}|v_n|\leq 
			\\
		\shoveright{\leq\frac{\|\Lambda(d)\|_\infty}{|x_k|^2\sin\alpha}\frac{\sin\alpha}8
			\big(\Im v_{n_k}\Im v_{n_k+1}\big)^{\frac 12}\leq\frac 1{8|x_k|}\|d\|_\infty
			}\\[-7mm]
	\end{multline*}
	\begin{multline*}
		\Big|\sum_{n=n_k+1}^\infty\frac{\Lambda(d)_n}{\qu{v_n}-x}\Big|\leq\|\Lambda(d)\|_\infty
		\sum_{n=n_k+1}^\infty\frac 1{\Im v_n}\leq\frac 1{8|x_k|}\|d\|_\infty\\[-7mm]
	\end{multline*}
	It follows that 
	\[
		\Big|\frac{x_k\Psi(d)(x_k)}{E(x_k)}\Big|\geq
		\Big|\underbrace{\sum_{n=1}^{n_k}\Lambda(d)_n}_{=|d_k|}\Big|-\frac{\|d\|_\infty}4
		\,,
	\]
	and, taking the supremum over all $k\in\bb N$, 
	\[
		\|\Psi(d)\|_{\tilde{\mf m}}\geq\frac 34\|d\|_\infty
		\,.
	\]
	Again by equivalence of norms, this shows that $\Psi$ maps $\ell^\infty$ bicontinuously into 
	$\langle R_{\mf m}(\mc H)\cap\mc L,\|.\|_{\mf m}\rangle$. 
\end{proofof}

\begin{remark}{R25}
	With similar arguments as used in \thref{R21} for the case when the zeros of $\Theta$ approach 
	$\bb R$, one can prove the following statement: 
	\textit{Let $\mc H$ be a de~Branges space and let $\mc L=\mc H(E_1)\in\Sub^*\mc H$, $\dim\mc L=\infty$. 
	Moreover, let $D\subset\bb C^+\cup\bb R$ be such that 
	\[
		\limsup\limits_{\substack{|z|\to\infty\\ z\in D}} \frac{\Im z}{|z|}=0
		\,.
	\]
	Assume that there exists a sequence $(w_k)_{k\in\bb N}$, $w_k\in\bb C^+$, with 
	\[
		\lim_{k\to\infty}\frac{\Im w_k}{|w_k|}=0,\quad \sup_{n\in\bb N}|\Theta(w_n)|<1
		\,.
	\]
	Then there exists a bicontinuous embedding of $\ell^\infty$ into 
	$R_{\mf m}(\mc H)\cap\mc L$.}
	
	We will not go into details. 
\end{remark}

\begin{remark}{R26}
	It is interesting to analyze the construction of the respective embeddings in the proofs of 
	\thref{R7}, and \thref{R21}, Case 1 / Case 2:
	\begin{center}
	\framebox{
	\begin{tabular}{ll}
		\thref{R7}: & $\xymatrix{\ell^\infty \ar[r]^\Lambda & \ell^2 \ar[r]^\mu & \ell^2 \ar[r]^\rho & \mc H}$\\
		\thref{R21}, Case 2: & $\xymatrix{\ell^\infty \ar[r]^\Lambda & \ell^\infty \ar[r]^\mu & \ell^1 \ar[r]^\rho & \mc H}$
			\\[2mm] \hline
		\thref{R21}, Case 1: \rule{0pt}{6mm} & 
			$\xymatrix{\ell^\infty \ar[r]^\mu & \ell^\infty \ar[r]^\gamma & \ell^1 \ar[r]^\rho & \mc H}$
	\end{tabular}
	}
	\end{center}
	In \thref{R21}, Case 2, the rough argument that $\rho\in\mc B(\ell^1,\mc H)$ was sufficient, whereas in \thref{R7} we had to 
	use the much finer argument that we can extract a Riesz sequence of normalized kernels. This is necessary, since in \thref{R7} 
	we only assume that $\Im w_n$ is bounded away from zero, and not that the points $w_n$ are contained in some Stolz angle. 
\end{remark}

\section{The majorant $\mf m^\flat$}

In this section we investigate the question in how many ways a given dB-subspace can be realized 
by majorization (provided it can be realized at all). In particular, we ask for "small" majorants which do the job. 
We will also view representability by different majorants from a refined viewpoint, requiring instead of 
$\mc R_{\mf m_1}(\mc H)=\mc R_{\mf m_2}(\mc H)$ that $R_{\mf m_1}(\mc H)=R_{\mf m_2}(\mc H)$ or even 
$B_{\mf m_1}(\mc H)=B_{\mf m_2}(\mc H)$. 

To start with, let us show by examples that in 
general there will actually exist many majorants $\mf m$ generating the same 
dB-subspace $\mc R_{\mf m}(\mc H)$. 

\begin{example}{R27}
	Consider a Paley--Wiener space $\PW_a$, $a>0$, and majorization on 
	$D:=\bb R$. For every $\mf m\in\Adm_{\bb R}\PW_a$ which is 
	separated from zero on each compact interval, we have 
	$\mc R_{\mf m}(\PW_a)=\PW_a$, cf.\ \cite[Corollary 3.12]{baranov.woracek:dbmaj}. 
	However, for each $\alpha\in (0,1)$, the function 
	\[
		\mf m(x):=\exp(-|x|^\alpha), \qquad x\in\bb R
		\,,
	\]
	is an admissible majorant, cf.\ \cite[Example 2.14]{baranov.woracek:dbmaj}. 
\end{example}

This example relies on the fact that no proper dB-subspaces of $\PW_a$ can be obtained by majorization 
on the real line. The following example is of a different nature. 

\begin{example}{R28}
	Let $\mc H$ be a de~Branges space and let $\mc L\in\Sub^*\mc H$, $\dim\mc L=\infty$, 
	be such that $\mc L=\clos_{\mc H}\bigcup\{\mc K\in\Sub^*\mc H:\,\mc K\subsetneq\mc L\}$. 
	We will consider majorization on $D:=i[1,\infty)$, so we know that $\nabla_{\!\mc L}|_D$ belongs to 
	$\Adm_D\mc H$ and $\mc R_{\nabla_{\!\mc L}|_D}(\mc H)=\mc L$. 
	
	By the Baire Category Theorem we have $\mc L\neq\bigcup\{\mc K\in\Sub^*\mc H:\,\mc K\subsetneq\mc L\}$. 
	Note here that this union is actually equal to some at most countable union. Moreover, 
	each dB-subspace is invariant with respect to $F\mapsto F^\#$. Hence we can choose 
	$F\in\mc L\setminus\bigcup\{\mc K\in\Sub^*\mc H:\,\mc K\subsetneq\mc L\}$, $\|F\|=1$, with $F=F^\#$. 
	
	It follows from a general argument that the function $F$ has infinitely many zeros: 
	Assume on the contrary that $w_1,\ldots,w_n$ are 
	all the zeros of $F$ (listed according to their multiplicities). Then the function 
	\[
		\tilde F(z):=F(z)\prod_{k=1}^n (z-w_k)^{-1}
	\]
	belongs to $\mc H$, satisfies $\tilde F^\#=\tilde F$, and has no zeros. 
	Thus $\spn\{\tilde F\}\in\Sub^*\mc H$. Hence also each of the 
	spaces 
	\[
		\mc L_m:=\big\{p \tilde F:\,p\in\bb C[z],
		\deg p\leq m\big\}\cap\mc H,\qquad m\in\bb N,\ m\le n,
	\]
	belongs to $\Sub^*\mc H$. By the construction of $\tilde F$, 
	we have $F\in\mc L_n$. Thus $\mc L_n$ is not contained 
	in any of the spaces $\mc K\in\Sub^*\mc H$, $\mc K\subsetneq\mc L$, 
	and, by de~Branges' Ordering Theorem, must therefore contain each of them. 
	By our assumption, $\bigcup_{\mc K\subsetneq\mc L}\mc K$ is dense in $\mc L$, 
	and it follows that also $\mc L\subseteq\mc L_n$. 
	This contradicts the assumption that $\dim\mc L=\infty$. 

	Denote the sequence of zeros of $F$ which lie in $\bb C^+\cup\bb R$ 
	by $(w_n)_{n\in\bb N}$, and let $B(z)$ be the Blaschke product 
	build from the zeros $w_n$ with positive imaginary part. Define 
	\[
		\mf m_k(z):=\frac{|F(z)|}{|z+i|^k|B(z)|},\qquad 
		z\in D,\ k\in\bb N_0 \,.
	\]
	Clearly, we have $\mf m_k>0$ and $\mf m_0\geq\mf m_1\geq\mf m_2\geq\ldots$. Moreover, for each $k\in\bb N_0$, 
	\[
		F_k(z):=\frac{F(z)}{\prod_{n=1}^k(z-\overline w_n)
		\cdot B(z)}\in R_{\mf m_k}(\mc H)\setminus
		R_{\mf m_{k+1}}(\mc H)
		\,.
	\]
	We see that $\mf m_k\in\Adm_D\mc H$. The same argument as used in \cite[Lemma 3.5, $(i)$]{baranov.woracek:fdimmaj} gives 
	\[
		\dim\mc R_{\mf m_k}(\mc H)\big/\mc R_{\mf m_{k-1}}(\mc H)\leq 1
		\,.
	\]
	Since each dB-subspace is closed with respect to multiplication by
	Blaschke products, the function $F_0$ does not belong to any 
	of the spaces $\mc K\in\Sub^*\mc H$, $\mc K\subsetneq\mc L$. Hence, 
	$\mc R_{\mf m_0}(\mc H)\supseteq\mc L$. On the other hand, $F_0\in\mc L$ and thus $\mf m_0\lesssim\nabla_{\!\mc L}$. 
	This implies that $\mc R_{\mf m_0}(\mc H)=\mc L$, and altogether we obtain $\mc R_{\mf m_k}(\mc H)=\mc L$, 
	$k\in\bb N_0$. 
	
	We have obtained a decreasing family $\mf m_k$ of majorants with 
	\[
		\mc R_{\mf m_k}(\mc H)=\mc L\text{ but }
		R_{\mf m_{k+1}}(\mc H)\subsetneq R_{\mf m_k}(\mc H),\qquad k\in\bb N_0
		\,.
	\]
	Let us remark moreover that $\lim_{k\to\infty}\mf m_k(z)=0$, $z\in D$, and thus also 
	\[
		\bigcap_{k\in\bb N}B_{\mf m_k}(\mc H)=\{0\}
		\,.
	\]
\end{example}

In particular, these examples show that the space $R_{\mf m}(\mc H)$ might be very small, despite the fact that 
$\mc R_{\mf m}(\mc H)$ is always the same. 

Let $\mf m\in\Adm\mc H$ be given. Among all those majorants $\mf m_1$ for which even 
$B_{\mf m_1}(\mc H)=B_{\mf m}(\mc H)$ there is a natural "sharp" majorant. 

\begin{definition}{R29}
	Let $\mc H$ be a de~Branges space and $\mf m\in\Adm\mc H$. 
	Define a function $\mf m^\flat:\bb C^+\cup\bb R\to[0,\infty)$ by 
	\[
		\mf m^\flat(z):=\sup\big\{|F(z)|:\,F\in B_{\mf m}(\mc H)\big\},\quad z\in\bb C^+\cup\bb R
		\,.
	\]
\end{definition}

\begin{lemma}{R30}
	Let $\mc H$ be a de~Branges space, $D\subseteq\bb C^+\cup\bb R$, and $\mf m\in\Adm_D\mc H$. 
	Then $\mf m^\flat\in\Adm\mc H$ and 
	\[
		B_{\mf m^\flat}(\mc H)=B_{\mf m^\flat|_D}(\mc H)=B_{\mf m}(\mc H)
		\,.
	\]
	Moreover, $\mf m^\flat|_D\leq\mf m$ and $\mf m^{\flat\flat}=\mf m^\flat$. 
\end{lemma}
\begin{proof}
	The inclusion $B_{\mf m^\flat}(\mc H)\subseteq B_{\mf m^\flat|_D}(\mc H)$
	is trivial. By the definition of $\mf m^\flat$, we have 
	$\mf m^\flat|_D\leq\mf m$, and therefore $B_{\mf m^\flat|_D}(\mc H)\subseteq B_{\mf m}(\mc H)$. 
	Let $F\in B_{\mf m}(\mc H)$ be given. Then $|F(z)|\leq\mf m^\flat(z)$, $z\in\bb C^+\cup\bb R$, and hence 
	$F\in B_{\mf m^\flat}(\mc H)$. 
	Finally, we have 
	\[
		\mf m^{\flat\flat}(z)=\sup_{F\in B_{\mf m^\flat}(\mc H)}|F(z)|=\sup_{F\in B_{\mf m}(\mc H)}|F(z)|=\mf m^\flat(z)
		\,.
	\]
\end{proof}

\begin{remark}{R36}
	We see that, given $\mf m\in\Adm_D\mc H$, the majorant $\mf m^\flat$ is the smallest among all 
	majorants $\mf m_1\in\Adm\mc H$ with $B_{\mf m_1}(\mc H)=B_{\mf m}(\mc H)$. 
\end{remark}

\leref{R30} can often be used to reduce considerations to majorants of the form $\mf m^\flat$. 
In view of this, it is worth mentioning that the majorant $\mf m^\flat$ 
is always fairly smooth and preserves real zeros as well as exponential growth. 

\begin{proposition}{R37}
	Let $\mc H$ be a de~Branges space, $D\subseteq\bb C^+\cup\bb R$, 
	and $\mf m\in\Adm_D\mc H$. Then $\mf m^\flat$ is continuous on $\bb C^+\cup\bb R$ 
	and $\log\mf m^\flat$ is subharmonic on $\bb C^+$. Moreover, we have 
	$\mf d_{\mf m^\flat}=\max\{\mf d_{\mf m},\mf d_{\mc H}\}$ and 
	$\mt_{\mc H} \mf m^\flat = \mt_{\mc H} \mc R_{\mf m}(\mc H)$. 
\end{proposition}
\begin{proof}
	The unit ball of $\mc H$ is a locally bounded, and thus normal, family of entire functions. This shows that the subset 
	$B_{\mf m}(\mc H)$ is also a normal family of entire functions, and thus equicontinuous. 
	This implies that $\mf m^\flat$ is continuous. 
	The function $\log\mf m^\flat$ is the supremum
	of the subharmonic functions $\log|F(z)|$, $F\in B_{\mf m}(\mc H)$. Since $\log\mf m^\flat$ is 
	continuous on $\bb C^+$, this implies that $\log\mf m^\flat$ is subharmonic on $\bb C^+$. 
	
	Let $w\in\bb C$, $n\in\bb N_0$, be such that $n<\mf d_{\mc H}(w)$. Then
	\[
		\Big\{\frac{F(z)}{(z-w)^n}:\,F\in B_{\mf m}(\mc H)\Big\}
	\]
	is a normal family of entire functions, and hence equicontinuous. Thus, given $\epsilon>0$, there
	exists $r>0$ such that
	\[
		\Big|\frac{F(z)}{(z-w)^n}\Big|\leq\epsilon,\quad |z-w|\leq r, \ F\in B_{\mf m}(\mc H)
		\,.
	\]
	This implies that also $|z-w|^{-n}\mf m^\flat(z)\leq\epsilon$ for $0<|z-w|\leq r$,
	and hence that $n<\mf d_{\mf m^\flat}(w)$. We conclude that $\mf d_{\mf m^\flat}(w)\geq\mf d_{\mc H}(w)$, $w\in\bb C$. 
	Now we obtain from \cite[(3.1)]{baranov.woracek:spmaj} that
	\[ 
		\mf d_{\mf m^\flat} = \max\{\mf d_{\mf m^\flat},\mf d_{\mc H}\}=
		\mf d_{\mc R_{\mf m^\flat}(\mc H)}= \mf d_{\mc R_{\mf m}(\mc H)}=\max\{\mf d_{\mf m},\mf d_{\mc H}\}
		\,.
	\]
	Since $B_{\mf m}(\mc H)$ is contained in the unit ball (with respect to the norm of $\mc H$) 
	of $\mc R_{\mf m}(\mc H)$, we have $\mf m^\flat\leq\nabla_{\!\mc R_{\mf m}(\mc H)}$. Hence, 
	by \cite[(3.1)]{baranov.woracek:spmaj}, $\mt_{\mc H} \mf m^\flat \leq \mt_{\mc H} \mc R_{\mf m}(\mc H)$. 
	Conversely, by \eqref{R43}, 
	\[
		\mt_{\mc H}\mc R_{\mf m}(\mc H) =
		\sup_{F\in R_{\mf m}(\mc H)}\mt_{\mc H} F 
		= \sup_{F\in B_{\mf m}(\mc H)}\mt_{\mc H} F 
		\leq \mt_{\mc H} \mf m^\flat
		\,.
	\]
	Thus $\mt_{\mc H} \mf m^\flat \geq \mt_{\mc H} \mc R_{\mf m}(\mc H)$.
\end{proof}

The next result says that equality of two spaces $R_{\mf m_1}(\mc H)$ and $R_{\mf m_2}(\mc H)$ can be 
characterized via $\mf m_1^\flat$ and $\mf m_2^\flat$. It is again a consequence of the completeness of $R_{\mf m}(\mc H)$. 

\begin{proposition}{R31}
	Let $\mc H$ be a de~Branges space, $D\subseteq\bb C^+\cup\bb R$, and $\mf m_1,\mf m_2\in\Adm_D\mc H$. 
	Then the following equivalences hold: 
	\begin{equation}\label{R32}
		\begin{aligned}
		R_{\mf m_1}(\mc H)\subseteq R_{\mf m_2}(\mc H)
		& \ \Longleftrightarrow \ 
		\exists\,\lambda>0:\,B_{\mf m_1}(\mc H)\subseteq \lambda B_{\mf m_2}(\mc H) \\
		& \ \Longleftrightarrow \ 
		\mf m_1^\flat\lesssim\mf m_2^\flat;
		\end{aligned}
	\end{equation}
	\begin{equation}\label{R33}
		\begin{aligned}
		R_{\mf m_1}(\mc H)=R_{\mf m_2}(\mc H)
			& \ \Longleftrightarrow \ 
			R_{\mf m_1}(\mc H)=R_{\mf m_2}(\mc H)\text{ and }\|.\|_{\mf m_1}\asymp\|.\|_{\mf m_2} \\
			& \ \Longleftrightarrow \ 
			\exists\,\lambda_1,\lambda_2>0:\,\lambda_1 B_{\mf m_1}(\mc H)\subseteq B_{\mf m_2}(\mc H)
			\subseteq\lambda_2 B_{\mf m_1}(\mc H) \\
			& \ \Longleftrightarrow \ 
 		\mf m_1^\flat\asymp\mf m_2^\flat;
		\end{aligned}
	\end{equation}
	\begin{equation}\label{R34}
		\begin{aligned}
		B_{\mf m_1}(\mc H)=B_{\mf m_2}(\mc H)
			& \ \Longleftrightarrow \ 
			R_{\mf m_1}(\mc H)=R_{\mf m_2}(\mc H)\text{ and }\|.\|_{\mf m_1}=\|.\|_{\mf m_2} \\
			& \ \Longleftrightarrow \ 
		\mf m_1^\flat=\mf m_2^\flat.
		\end{aligned}
	\end{equation}
\end{proposition}
\begin{proof}
	Assume that $R_{\mf m_1}(\mc H)\subseteq R_{\mf m_2}(\mc H)$. Since $\|.\|_{\mf m_j}\geq\|.\|_{\mc H}$, point evaluation is 
	continuous with	respect to each of the norms $\|.\|_{\mf m_j}$. Hence the map 
	$\id:\langle R_{\mf m_1}(\mc H),\|.\|_{\mf m_1}\rangle\to\langle R_{\mf m_2}(\mc H),\|.\|_{\mf m_2}\rangle$ 
	has closed graph. By the Closed Graph Theorem, it is therefore bounded, i.e.\ there exists a positive constant 
	$\lambda$, such that $B_{\mf m_1}(\mc H)\subseteq\lambda B_{\mf m_2}(\mc H)$. 
	Next assume that $B_{\mf m_1}(\mc H)\subseteq\lambda B_{\mf m_2}(\mc H)$ with some $\lambda>0$. Then it follows that 
	\[
		\frac 1\lambda\mf m_1^\flat(z)=\sup_{F\in B_{\mf m_1}(\mc H)} \big|\frac 1\lambda F(z)\big|
		\leq\sup_{G\in B_{\mf m_2}(\mc H)}|G(z)|=\mf m_2^\flat(z)
		\,.
	\]
	Assume finally that $\mf m_1^\flat\lesssim\mf m_2^\flat$. If $F\in B_{\mf m_1}(\mc H)$, then 
	$F\in\mc H$ and 
	\begin{equation}\label{R35}
		|F(z)|\leq\mf m_1^\flat(z)\lesssim\mf m_2^\flat(z)\leq\mf m_2(z),\qquad z\in D
		\,.
	\end{equation}
	Hence $F\in R_{\mf m_2}(\mc H)$, and we conclude that $R_{\mf m_1}(\mc H)\subseteq R_{\mf m_2}(\mc H)$. 
	This finishes the proof of \eqref{R32}. 
	
	The equivalences in \eqref{R33} are an immediate consequence of \eqref{R32}. Those in \eqref{R34} are seen 
	by similar arguments as above, noting that $\mf m_1^\flat=\mf m_2^\flat$ gives a more accurate 
	estimate in \eqref{R35}. 
\end{proof}

When investigating "small" majorants, naturally the question comes up whether there exist minimal ones among all 
admissible majorants. First, let us make precise what we understand by the terms "small" or "minimal". 

\begin{definition}{R38}
	On the set of all pairs $(\mf m,D)$, where $D\subseteq\bb C^+\cup\bb R$ and $\mf m:D\to[0,\infty)$, 
	we define a relation $\preceq$ by
	\[
		(\mf m_1,D_1)\preceq(\mf m_2,D_2)\ \iff\ D_1\supseteq D_2\text{ and }\mf m_1|_{D_2}\lesssim\mf m_2
		\,.
	\]
\end{definition}

Clearly, the relation $\preceq$ is reflexive and transitive, i.e.\ $\Adm\mc H$ is preordered by $\preceq$. 
Moreover, $(\mf m_1,D_1)\preceq(\mf m_2,D_2)$ and
$(\mf m_2,D_2)\preceq(\mf m_1,D_1)$ both hold at the same time if and only if
\[
	D_1=D_2\quad\text{and}\quad \mf m_1\asymp\mf m_2
	\,.
\]
Whenever we speak of order-theoretic terms in the context of majorization, we refer to the order induced 
by $\preceq$. 

The validity of $\mf m_1\preceq\mf m_2$ means that majorization by $\mf m_1$ is a stronger requirement than 
majorization by $\mf m_2$. In fact, $(\mf m_1,D_1)\preceq(\mf m_2,D_2)$ implies 
\[
	R_{\mf m_1}(\mc H)\subseteq R_{\mf m_2}(\mc H)
	\,.
\]
Let us note that in general the converse does not hold, even if $D_1=D_2$. 

\begin{example}{R39}
	Assume that $\mc H=\mc H(E)$ contains the set of all polynomials
	$\bb C[z]$ as a dense linear subspace. Such de Branges subspaces were 
	studied in \cite{baranov:2006}; in particular, whenever $E$ is a canonical product whose zeros all lie on the 
	imaginary axis and have genus zero, the space $\mc H(E)$ will have this property. 

	Let $n\in\bb N$, and set $\mf m_1(z)=(1+|z|)^{n+1/2}$ and $\mf m_2(z)=(1+|z|)^n$, $z\in\bb C^+\cup\bb R$. 
	Then $R_{\mf m_1}(\mc H)=R_{\mf m_2}(\mc H)$ equals 
        the set of all polynomials whose 
	degree does not exceed $n$. Moreover, the norms $\|.\|_{\mf m_1}$
	and $\|.\|_{\mf m_2}$ are equivalent on $R_{\mf m_1}(\mc H)$ and so
	there is $\delta>0$ such that 
	$B_{\delta \mf m_1}(\mc H)\subseteq B_{\mf m_2}(\mc H)$.
	However, it is not true that $(\delta \mf m_1,\bb C^+\cup \bb R)
        \preceq(\mf m_2,\bb C^+\cup \bb R)$.
\end{example}

The following result is an extension of \cite[Theorem 4.2]{baranov.woracek:dbmaj} where the case of majorization 
along $\bb R$ was treated. For a de~Branges space $\mc H$ denote by $\mf r_{\mc H}:\Adm\mc H\to\Sub\mc H$ the map 
\[
	\mf r_{\mc H}(\mf m):=\mc R_{\mf m}(\mc H)
	\,,
\]
and let $\Min\Adm_D\mc H$ be the set of all minimal elements of $\Adm_D\mc H$ modulo $\asymp$. 

\begin{proposition}{R40}
	Let $\mc H$ be a de~Branges space, and let $D\subseteq\bb C^+\cup\bb R$. Then $\mf r_{\mc H}$ 
	maps $\Min\Adm_D\mc H$ bijectively onto 
	\[
		\mf L:=\big\{\mc L\in\mf r_{\mc H}(\Adm_D\mc H):\,\dim\mc L=1\big\}
		\,.
	\]
	If $\mc L\in\mf L$, then $(\mf r_{\mc H}|_{\Min\Adm_D\mc H})^{-1}(\mc L)\asymp\nabla_{\!\mc L}|_D$. 
\end{proposition}
\begin{proof}
	Assume that $\mf m\in\Adm_D\mc H$ is minimal, but $\dim R_{\mf m}(\mc H)>1$. Then there exists an element 
	$F\in R_{\mf m}(\mc H)$ with $F(i)=0$. It follows that 
	\[
		\frac{F(z)}{z-i}\in R_{\tilde{\mf m}}(\mc H)
	\]
	where $\tilde{\mf m}(z):=(1+|z|)^{-1}\mf m(z)$, $z\in D$. This yields that $\tilde{\mf m}\in\Adm_D\mc H$. However, 
	$\tilde{\mf m}\preceq\mf m$ but $\mf m\not\asymp\tilde{\mf m}$, and we have obtained a contradiction. We conclude 
	that $\mf r_{\mc H}$ maps $\Min\Adm_D\mc H$ into $\mf L$. 
	
	We proceed with an intermediate remark: 
	Let $\mc L$ be any one-dimensional dB-subspace of $\mc H$. Fix $F_0\in\mc L\setminus\{0\}$, then $\mc L=\spn\{F_0\}$ and thus 
	$\nabla_{\!\mc L}(z)=\|F_0\|_{\mc H}^{-1}|F_0(z)|$. If $\mf m\in\Adm_D\mc H$ has the property that 
	$\mc L=\mc R_{\mf m}(\mc H)$, then 
	\[
		B_{\mf m}(\mc H)=\big\{\lambda F_0:\,|\lambda|\leq\|F_0\|_{\mf m}\big\}\ \text{ and }\ 
		\mf m^\flat(z)=\frac{|F_0(z)|}{\|F_0\|_{\mf m}}
		\,.
	\]
	It follows that $\mf m^\flat\asymp\nabla_{\!\mc L}|_D$. 
	
	By \leref{R30} a minimal majorant $\mf m$ satisfies $\mf m\asymp\mf m^\flat|_D$. Hence, 
	together with the first paragraph of this proof, this remark already shows that 
	$\mf r_{\mc H}|_{\Min\Adm_D\mc H}$ is injective. To see surjectivity, let 
	$\mc L\in\mf L$ be given. Then the above remark yields $\mc L=\mc R_{\nabla_{\!\mc L}|_D}(\mc H)$. 
	If $\mf m\in\Adm_D\mc H$ and $\mf m\preceq\nabla_{\!\mc L}|_D$, then 
	$\{0\}\neq\mc R_{\mf m}(\mc H)\subseteq\mc R_{\nabla_{\!\mc L}|_D}(\mc H)=\mc L$, and hence also 
	$\mc R_{\mf m}(\mc H)=\mc L$. It follows that $\mf m^\flat|_D\asymp\nabla_{\!\mc L}|_D$. 
	Since $\mf m^\flat|_D\lesssim\mf m$, we obtain that also $\mf m\asymp\nabla_{\!\mc L}|_D$. 
	Thus $\nabla_{\!\mc L}|_D\in\Min\Adm_D\mc H$. 
\end{proof}

In conjunction with the representability results shown in \cite{baranov.woracek:spmaj}, we obtain the 
following analogue of \cite[Theorem 4.9]{baranov.woracek:dbmaj} for majorization on the imaginary half-line. 

\begin{corollary}{R41}
	Let $\mc H$ be a de~Branges space, and set $D:=i[1,\infty)$. Then the set $\Adm_D\mc H$ contains a 
	minimal element	if and only if $\Sub^*\mc H$ contains a one-dimensional subspace $\mc L_0$. 
	In this case there exists exactly one minimal element, namely $\nabla_{\!\mc L_0}|_D$. 
\end{corollary}
\begin{proof}
	By \cite[Theorem 4.1]{baranov.woracek:spmaj}, we have $\mf r_{\mc H}(\Adm_D\mc H)=\Sub^*\mc H$. By de~Branges' Ordering Theorem, the set 
	$\Sub^*\mc H$ can contain at most one one-dimensional subspace. 
\end{proof}



{\footnotesize
\begin{flushleft}
	A.\,Baranov\\
	Department of Mathematics and Mechanics\\
	Saint Petersburg State University\\
	28, Universitetski pr.\\
	198504 Petrodvorets\\
	RUSSIA\\
	email: a.baranov@ev13934.spb.edu\\[5mm]
\end{flushleft}
\begin{flushleft}
	H.\,Woracek\\
	Institut f\"ur Analysis und Scientific Computing\\
	Technische Universit\"at Wien\\
	Wiedner Hauptstr.\ 8--10/101\\
	A--1040 Wien\\
	AUSTRIA\\
	email: harald.woracek@tuwien.ac.at\\[5mm]
\end{flushleft}
}

\end{document}